\documentclass[12pt,reqno]{amsart}
\usepackage{amscd,amsmath,amsopn,amssymb,amsthm,multicol}
\usepackage{hyperref}
\usepackage[all]{xy}
\usepackage{amscd}
\usepackage{graphicx}

\DeclareMathOperator{\Ad}{Ad}
\DeclareMathOperator{\ad}{ad}
\DeclareMathOperator{\Aut}{Aut}

\DeclareMathOperator{\Ric}{Ric}

\newcommand{\fr}{\mathfrak}
\newcommand{\al}{\alpha}
\newcommand{\be}{\beta}

\newcommand{\thickline}{\noalign{\hrule height 1pt}}

\newtheorem{theorem}{Theorem}
\newtheorem{lemma}{Lemma}

\newtheorem{prop}{Proposition}

\begin{document}

\title
{Complete description of invariant Einstein metrics on the generalized flag manifold $SO(2n)/U(p)\times U(n-p)$}

\author{Andreas Arvanitoyeorgos, Ioannis Chrysikos and Yusuke Sakane}
\address{University of Patras, Department of Mathematics, GR-26500 Rion, Greece}
\email{arvanito@math.upatras.gr}
\email{xrysikos@master.math.upatras.gr}
\address{Osaka University, Department of Pure and Applied Mathematics, Graduate School of Information and Technology,
Osaka 560-043, Japan}
\email{sakane@math.sci.osaka-u.ac.jp}
 
\medskip
\noindent
\thanks{The first two authors  were partially supported
  by the C.~Carath\'{e}odory grant \#C.161 2007-10,
  University of Patras and the third auther was supported by Grant-in-Aid
for Scientific Research (C) 21540080}

\begin{abstract}
 We find the precise number of non-K\"ahler $SO(2n)$-invariant Einstein metrics on the generalized flag manifold
$M=SO(2n)/U(p)\times U(n-p)$ with $n\geq 4$ and $2\leq p\leq n-2$.
We use an analysis on parametric systems of polynomial equations and we give some insight towards the study of such systems. 
  We also examine the isometric problem 
 for these Einstein metrics.

 \medskip
\noindent 2000 {\it Mathematics Subject Classification.} Primary 53C25; Secondary 53C30, 12D05, 65H10

\medskip
\noindent {\it Keywords}:  homogeneous  manifold, Einstein metric,  generalized flag manifold,  algebraic systems
of equations, resultant.

\end{abstract}

\maketitle

\section*{Introduction}
\markboth{Andreas Arvanitoyeorgos, Ioannis Chrysikos and Yusuke Sakane}{Invariant Einstein metrics on the generalized flag manifold $SO(2n)/U(p)\times U(n-p)$}

A Riemannian metric $g$    is called {\it Einstein} if the Ricci tensor
 $\Ric_{g}$ satisfies the equation  
 ${\rm  Ric}_{g}=e\cdot g$, 
  for some $e\in\mathbb{R}$.     When $M$ is compact, 
  Einstein metrics of   volume 1 
  can be characterized variationally  as the critical 
  points of the scalar curvature functional 
    $T(g)=\int_{M}S_{g}d {\rm vol}_{g}$ 
  on the space  $\mathcal{M}_{1}$ of Riemannian metrics of 
  volume 1. If  $M=G/K$  is a compact homogeneous space, a $G$-invariant
Einstein metric  is precisely a critical 
point  of $T$ restricted to
the set  of $G$-invariant metrics of volume 1.
As a consequence, the Einstein equation reduces to a system 
of non-linear algebraic equations, 
which is still very complicated but more manageable, 
and in some times can be solved explicity. 
 Thus most known examples of Einstein manifolds are homogeneous.

  In a recent work \cite{Chry3} the first two authors classified 
  all flag manifolds for which the isotropy representation decomposes 
  into four pairwise inequivalent irreducible submodules, and found new 
  invariant Einstein metrics on these spaces.  Recall that a generalized 
  flag manifold is an adjoint orbit  of a compact semisimple Lie group $G$, 
  or equivalently a compact homogeneous space  of the form  $M=G/K=G/ C(S)$, 
   where $C(S)$ is the centralizer  of a torus $S$ in $G$.

Eventhough the problem of finding all invariant 
Einstein metrics on $M$  can be facilitated by use of certain
theoretical results (e.g. the work \cite{Gr} on 
the total number of $G$-invariant complex Einstein metrics),
it still remains a difficult one, especially when the number 
of isotropy summands increases.  This difficulty also increases 
when we pass from exceptional flag manifolds to classical flag 
manifolds, because in the later case the Einstein equation reduces 
to a parametric system.   In  particular,  eventhough all invariant 
Einstein metrics were found for every generalized flag manifold with
 four isotropy summands,
a partial answer was given for the spaces $SO(2n)/U(p)\times U(n-p)$ 
 and $Sp(n)/U(p)\times U(n-p)$.

We summarize  the results obtained in \cite{Chry3} about these spaces.

\begin{theorem}\label{The1}{\textnormal{(\cite{Chry3})}}
The flag manifold $SO(2n)/U(p)\times U(n-p)$ ($n\geq 4$ and $2\leq p\leq n-2$)  
    admits at least six 
    $SO(2n)$-invariant Einstein metrics.  There are  two   non-K\"ahler 
     Einstein metrics and  two pairs of isometric K\"ahler-Einstein metrics.
\end{theorem} 

\begin{theorem}\label{The2}{\textnormal{(\cite{Chry3})}}
The flag manifold $Sp(n)/U(p)\times U(n-p)$ ($n\geq 2$ and $1\leq p\leq n-1$)  
    admits at least four 
    $Sp(n)$-invariant Einstein metrics, which are K\"ahler. 
\end{theorem} 

For the special case $n=2p$ the following results have been obtained:

\begin{theorem}\label{The3}{\textnormal{(\cite{Chry3})}}
The flag manifold $SO(4n)/U(p)\times U(p)$ ($p\geq 2$)  
    admits at least six 
    $SO(4n)$-invariant Einstein metrics. There are two  non-isometric non-K\"ahler  
    Einstein metrics,   and four isometric  K\"ahler-Einstein metrics.
    In the special case where $2\le p\le 6$ there are two more non-K\"ahler 
      Einstein metrics, and the
    total number of $SO(4n)$-invariant Einstein metrics is exactly eight.
\end{theorem} 

\begin{theorem}\label{The4}{\textnormal{(\cite{Chry3})}}
The flag manifold $Sp(2n)/U(p)\times U(p)$ ($p\geq 1$)  
    admits precisely six 
    $Sp(n)$-invariant Einstein metrics.  There are four isometric K\"ahler-Einstein 
    metrics, and two   non-K\"ahler  Einstein metrics. 
\end{theorem}

In the present paper we find all $SO(2n)$-invariant Einstein 
metrics on the flag manifold $SO(2n)/U(p)\times U(n-p)$, by using
a new approach into manipulating the algebraic systems 
of equations obtained from the Einstein equation.
The coefficients of the polynomials in such systems involve parameters, so   a 
major difficulty appears when we try to show  existence and uniqueness of solutions. 
Therefore, the contribution of the present work is,
 besides answering the original problem on Einstein metrics, 
to give some insight towards the study of parametric systems of algebraic equations. 
 
Our main result is the following:

\medskip
{ \sc{Main Theorem.}}
    {\it Let $M=SO(2n)/U(p)\times U(n-p)$ with $n\geq 4$ and $2\leq p\leq n-2$.
    Then $M$ admits exactly four non-K\"ahler $SO(2n)$-invariant Einstein metrics for the pairs
    $(n,p)=(12, 6)$, $(10, 5)$, $(8, 4)$, $(7, 4)$, $(7, 3)$, 
    $(6, 4)$,  $(6, 3)$, $(6, 2)$, $(5, 3)$, $(5, 2)$, $(4, 2)$, and
    two non-K\"ahler $SO(2n)$-invariant Einstein metrics for all other cases.} 
  
  \medskip  
 The flag manifold $Sp(n)/U(p)\times U(n-p)$ will be treated in a forthcoming paper.

  \section{The Einstein equation on  flag manifolds}

Let $M=G/K=G/C(S)$ be a generalized flag manifold of a compact  
 simple Lie group $G$, where   $K=C(S)$ is the centralizer 
   of a torus $S$ in $G$.      
  Let   $o=eK$  be the identity coset of $G/K$. 
We denote by $\fr{g}$ and $\fr{k}$ the corresponding Lie 
algrebras of $G$ and $K$. Let $B$ denote the  
Killing form of $\fr{g}$.  Since $G$ is compact and  simple,  $-B$ is a positive 
definite inner product on $\fr{g}$.  With repsect to $-B$ we consider the orthogonal 
decomposition $\fr{g}=\fr{k}\oplus\fr{m}$. 
This is a reductive decomposition 
of $\fr{g}$, that is $\Ad(K)\fr{m}\subset\fr{m}$, and as usual we identify the tangent 
space $T_{o}M$   with $\fr{m}$.  Since $K=C(S)$, the isotropy group $K$ is connected and   the relation  $\Ad(K)\fr{m}\subset\fr{m}$ is equivalent with  $[\fr{k}, \fr{m}]\subset\fr{m}$.  Thus,  for a flag manifold $M=G/K$ the notion of $\Ad(K)$-invariant and $\ad(\fr{k})$-invariant is equivalent.

Let  $\chi : K\to \Aut(T_{o}M)$ 
be the isotropy representation of $K$ on $T_{o}M$. 
Since $\chi$ is equivalent to the adjoint representation   
of $K$ restricted on $\fr{m}$,  the set of  all $G$-invariant symmetric covariant  2-tensors on $G/K$  
can be identified with the set of all $\Ad(K)$-invariant symmetric bilinear forms on $\fr{m}$.  
In particular, the set of  $G$-invariant metrics on $G/K$ 
is identified with the set of $\Ad(K)$-invariant inner products  on $\fr{m}$.  

Let $\fr{m}=\fr{m}_1\oplus\cdots\oplus\fr{m}_{s}$ be a $(-B)$-orthogonal 
$\Ad(K)$-invariant decomposition of $\fr{m}$ into  pairwise inequivalent  irreducible 
$\Ad(K)$-modules $\fr{m}_{i}$  $(i=1, \ldots, s)$.  Such a decomposition always exists and can be expressed in terms of $\fr{t}$-roots (cf. \cite{AP}, \cite{Chry3}).
Then, a $G$-invariant Riemannian metric  on $M$ (or equivalently, an $\Ad(K)$-invariant inner product $\langle \ , \ \rangle$ on $\fr{m}=T_{o}M$) is given by 
\begin{equation}\label{Inva}
g=\langle \ , \ \rangle =x_1\cdot (-B)|_{\fr{m}_1}+\cdots+x_s\cdot (-B)|_{\fr{m}_s},
\end{equation}
where  $(x_1, \ldots, x_s)\in\mathbb{R}^{s}_{+}$.  Since $\fr{m}_{i}\neq\fr{m}_{j}$ as $\Ad(K)$-representation, any $G$-invariant metric on $M$ has the above form. 

Similarly, the Ricci tensor $\Ric_{g}$ of a $G$-invariant metric $g$ on $M$, as a symmetric covariant 2-tensor on $G/K$ 
  is given by
  \[
\Ric_{g} = r_1\cdot (-B)|_{\fr{m}_1}+\cdots+r_s\cdot (-B)|_{\fr{m}_s}, 
\]
where $r_{1}, \ldots, r_{s}$  are the components of the Ricci 
tensor on each $\fr{m}_{i}$, that is $\Ric_{g}|_{\fr{m}_{i}}=r_{i}\cdot (-B)|_{\fr{m}_{i}}$.  These components obtain o useful description in terms of the structure constants $[ijk]$ first introduced in \cite{Wa2}.   
   Let $\{X_{\al}\}$ be a $(-B)$-orthonormal basis adapted to the 
   decomposition of $\fr{m}$, that is  $X_{\al}\in \fr{m}_{i}$ for some $i$, 
   and $\al<\be$ if $i<j$ (with $X_{\al}\in \fr{m}_{i}$ and $X_{\be}\in\fr{m}_{j}$).  
   Set $A_{\al\be}^{\gamma}=B([X_{\al}, X_{\be}], X_{\gamma})$ so 
   that $[X_{\al}, X_{\be}]_{\fr{m}}=\sum_{\gamma}A_{\al\be}^{\gamma}X_{\gamma}$, 
   and   $[ijk]=\sum(A_{\al\be}^{\gamma})^{2}$, where the sum is taken over all 
   indices $\al, \be, \gamma$ with $X_{\al}\in \fr{m}_{i}, X_{\be}\in\fr{m}_{j}, X_{\gamma}\in\fr{m}_{k}$
   (where $[ \ , \ ]_{\fr{m}}$ denotes the $\fr{m}$-component).  Then $[ijk]$ is nonnegative, symmetric in all 
   three entries, and independent of the $(-B)$-orthonormal bases choosen for 
   $\fr{m}_{i}, \fr{m}_{j}$ and $\fr{m}_{k}$ (but it depends on the choise 
   of the decomposition of $\fr{m}$).

 \begin{prop}\label{Ricc}{\textnormal{(\cite{SP})}} 
Let $M=G/K$ be a generalized flag manifold of a compact simple Lie group $G$ and let
 $\fr{m}=\bigoplus_{i=1}^{s}\fr{m}_{i}$ be a decomposition of $\fr{m}$ into pairwise
 inequivalent irreducible $\Ad(K)$-submodules.  
Then the components $r_{1}, \ldots, r_{s}$ of the Ricci tensor of a $G$-invariant 
metric  (\ref{Inva}) on $M$ are given by
   \[
   r_{k}=\frac{1}{2x_{k}}+\frac{1}{4d_{k}}\sum_{i, j}\frac{x_{k}}{x_{i}x_{j}}[ijk]-\frac{1}{2d_{k}}\sum_{i, j}\frac{x_{j}}{x_{k}x_{i}}[kij], \qquad (k=1, \ldots, s).
\]
\end{prop}

 In wiew of Proposition \ref{Ricc}, a $G$-invariant metric $g=(x_1, \ldots, x_{s})\in\mathbb{R}^{s}_{+}$ on $M$, is an Einstein metric with Einstein constant $e$, if and only if it is a positive real solution of the system 
 \[
 \frac{1}{2x_{k}}+\frac{1}{4d_{k}}\sum_{i, j}\frac{x_{k}}{x_{i}x_{j}}[ijk]-\frac{1}{2d_{k}}\sum_{i, j}\frac{x_{j}}{x_{k}x_{i}}[kij]=e, \quad 1\leq k\leq s.
  \]

\section{The generalized flag manifold $SO(2n)/U(p)\times U(n-p)$}

  We review some results related to the generalized flag manifold $M=G/K=SO(2n)/U(p)\times U(n-p)$ ($n\ge 4, \ 2\le p\le n-2$)
  obtained in \cite{Chry3}.
  Its corresponding painted Dynkin diagram is given by
 
 \medskip 
  
\hspace{4.5cm}{\begin{picture}(160,40)(-15,-23)
\put(0, 0){\circle{4}}
\put(0,8.5){\makebox(0,0){$\al_1$}}
\put(0,-8){\makebox(0,0){1}}
\put(2, 0){\line(1,0){14}}
\put(18, 0){\circle{4}}
\put(18,8.5){\makebox(0,0){$\al_2$}}
\put(18,-8){\makebox(0,0){2}}
\put(20, 0){\line(1,0){10}}
\put(40, 0){\makebox(0,0){$\ldots$}}
\put(50, 0){\line(1,0){10}}
\put(60, -19){\makebox(0,0){$( 2 \leq p \leq \ell -2 )$}}
\put(60, 0){\circle*{4.4}}
\put(60,8.5){\makebox(0,0){$\al_p$}}
\put(60,-8){\makebox(0,0){2}}
\put(60, 0){\line(1,0){10}}
\put(80, 0){\makebox(0,0){$\ldots$}}
\put(90, 0){\line(1,0){10}}
\put(102, 0){\circle{4}}
\put(102,-8){\makebox(0,0){2}}
\put(103.7, 1){\line(2,1){10}}
\put(103.7, -1){\line(2,-1){10}}
\put(115.5, 6){\circle{4}}
\put(115.5, -6){\circle*{4}}
\put(123.5, 14){\makebox(0,0){$\al_{\ell-1}$}}
\put(123.5,5.5){\makebox(0,0){1}}
\put(115.5, -16){$\al_\ell$}
\put(123.5,-8){\makebox(0,0){1}}
\end{picture}}

\medskip
  
  The isotropy representation of $M$ decomposes into a direct sum $\chi = \chi _1\oplus\chi _2\oplus\chi _3\oplus\chi _4$, which
  gives rise to a decomposition $\fr{m}=\fr{m}_1\oplus\fr{m}_2\oplus\fr{m}_3\oplus\fr{m}_4$ of $\fr{m}=T_{o}M$
  into   four irreducible inequivalent  $\ad(\fr{k})$-submodules. 
  The dimensions $d_i =\dim \fr{m}_i\ (i=1,2,3,4)$ of these
  submodules can be obtained by use of Weyl's formula \cite[pp.~204-205, p.~210]{Chry3}    
  and are given by
  $$d_1=2p(n-p),  \ d_2=(n-p)(n-p-1), \  d_3=2p(n-p),\   d_4=p(p-1).
  $$
  
According to (\ref{Inva}), a $G$-invariant metric on $M=G/K$ is given by
\begin{equation}\label{metrI}
\left\langle \ , \ \right\rangle=x_1\cdot (-B)|_{\fr{m}_1}+
x_2\cdot (-B)|_{\fr{m}_2}+x_3\cdot (-B)|_{\fr{m}_3}+x_4\cdot (-B)|_{\fr{m}_4},
\end{equation}
 for positive real numbers $x_1, x_2, x_3, x_4$.
We will denote such metrics by $g=(x_1, x_2, x_3, x_4)$.

 It is known (\cite{N}) that if $n\ne 2p$ then $M$ admits two non-equivalent $G$-invariant complex structures $J_1, J_2$, and thus two non-isometric  K\"ahler-Einstein metrics which are given (up to scale) by (see also \cite[Theorem 3]{Chry3}) 
 \begin{equation}\label{KE}
 \begin{tabular}{l}  
 $g_1= (n/2, \ n+p-1, \ n/2+p-1, \ p-1)$  \\
 $g_2= (n/2, \ n-p-1, \ 3n/2-p-1, \ 2n-p-1)$.
 \end{tabular}
 \end{equation}
 If $n=2p$ then $M$ admits a unique $G$-invariant complex structure with
 corresponding K\"ahler-Einstein metric (up to scale) given by
  $g= (p, \ p-1, \ 2p-1, \ 3p-1)$ (cf. also \cite[Theorem 8]{Chry3} where all isometric 
  K\"ahler-Einstein metrics are listed).

 The Ricci tensor of $M$ has been computed in \cite{Chry3} and is given as follows:
 
 \begin{prop}\label{componentsII}  The  components $r_i$ of the Ricci tensor
  for a   $G$-invariant Riemannian metric on $M$ determined by {\rm (\ref{metrI})} are given as follows:

\begin{equation}\label{compIII}
 \ \ \left.
\begin{tabular}{l}
$r_1=\displaystyle\frac{1}{2x_1}+  \frac{c_{12}^3}{2d_1}\Big( \frac{x_1}{x_2x_3}- \frac{x_2}{x_1x_3}- \frac{x_3}{x_1x_2}\Big)+  \frac{c_{13}^4}{2d_1}\Big( \frac{x_1}{x_3x_4}- \frac{x_4}{x_1x_3}- \frac{x_3}{x_1x_4}\Big)$ \\
$r_2=\displaystyle\frac{1}{2x_2}+  \frac{c_{12}^3}{2d_2}\Big( \frac{x_2}{x_1x_3}- \frac{x_1}{x_2x_3}- \frac{x_3}{x_1x_2}\Big)$ \\
$r_3=\displaystyle\frac{1}{2x_3}+  \frac{c_{12}^3}{2d_3}\Big( \frac{x_3}{x_1x_2}- \frac{x_2}{x_1x_3}- \frac{x_1}{x_2x_3}\Big)+  \frac{c_{13}^4}{2d_3}\Big( \frac{x_3}{x_1x_4}- \frac{x_4}{x_1x_3}- \frac{x_1}{x_3x_4}\Big)$ \\
$r_4 = \displaystyle\frac{1}{2x_4}+  \frac{c_{13}^4}{2d_4}\Big( \frac{x_4}{x_1x_3}- \frac{x_3}{x_1x_4}- \frac{x_1}{x_3x_4}\Big),$ \\
\end{tabular}\right\}
\end{equation}
where $c_{12}^3=[123]$ and $c_{13}^4=[134]$.
\end{prop}

 By taking into account the explicit form of the K\"ahler-Einstein metrics above,  
 and substituting these in  (\ref{compIII}), 
 we can find that the   values of the unknown triples $[ijk]$ are given by
$\displaystyle{ c_{12}^3=\frac{p(n-p)(n-p-1)}{2(n-1)}}$  and $\displaystyle{ c_{13}^4=\frac{p(p-1)(n-p)}{2(n-1)}}$.

\medskip

A $G$-invariant metric $g=(x_1, x_2, x_3, x_4)$ on $M=G/K$ is Einstein if and only if, there is a positive
constant $e$ such that $r_1=r_2=r_3=r_4=e$, or equivalently
\begin{equation}\label{systemI}
   r_1-r_3=0, \quad r_1-r_2=0, \quad r_3-r_4=0. 
\end{equation} 

By substituting the values of $d_i\ (i=1,2,3,4)$ and $c_{12}^3, c_{13}^4$ into the components of the Ricci tensor,
System (\ref{systemI}) is equivalent to the following equations:  

\begin{equation}\label{syst5}
 \left.
 \begin{tabular}{r}
$(x_1-x_3)(-x_1x_2+px_1x_2-x_2x_3+px_2x_3-x_1x_4+nx_1x_4$\\
 $-px_1x_4+2x_2x_4-2nx_2x_4-x_3x_4+nx_3x_4-px_3x_4)=0$\\
$4(n-1) x_3 x_4 (x_2-x_1)+ (n+p-1)x_4(x_1^2-x_2^2)- (n-3p-1)x_3^2x_4$\\
$ +(p-1)x_2(x_1^2-x_3^2-x_4^2)=0$\\
$4(n-1)x_1x_2(x_4-x_3)+ (2n-p-1)x_2(x_3^2-x_4^2)+ (2n-3p+1)x_1^2x_2$\\
$ +(n-p-1)x_4(x_3^2-x_1^2-x_2^2)=0$\\
\end{tabular}\right\} 
  \end{equation} 

\section{Proof of the Main Theorem}

We   consider  the equation $r_1- r_3 = 0$ of System (\ref{systemI}).  This is equivalent to  
  
\begin{eqnarray*}  
& (x_1 - x_3) (-x_1 x_2 + p x_1 x_2 - x_2 x_3 + p x_2 x_3 - x_1 x_4   + n x_1 x_4  & \\   
& -  p x_1 x_4 + 2 x_2 x_4    - 2 n x_2 x_4 - x_3 x_4 + n x_3 x_4 - p x_3 x_4) =0.  &
 \end{eqnarray*}

\noindent\underline{CASE A}
Let $x_1= x_3 =1$. Then  
   the system of  equations $r_1-r_2=0,\ r_3-r_4=0$ becomes 
\begin{eqnarray} 
  {x_2}^2 (n+p-1)+4 (n-p-1)-4 (n-1) {x_2}+(p-1) {x_2} {x_4} &= &0 \label{1} \\
{x_2} {x_4} (n-p-1)+{x_4}^2 (2 n-p-1)-4 (n-1) {x_4}+4 (p-1)  & =& 0. \label{2}
\end{eqnarray}
From (\ref{1}) we get that 
  \begin{eqnarray}  x_4 = -\frac{({x_2}-2)( (n+p-1){x_2}-2(n-p-1))}{(p-1){x_2}}.  \label{3}
  \end{eqnarray}
  Note that $x_4 > 0 $ if and only if $\displaystyle \frac{2(n-p-1)}{n+p-1} < x_2 < 2$. 
By substituting  equation (\ref{3})  into  equation (\ref{2}),  we obtain the following equation:  
  \begin{eqnarray} 
  & & \ \    H_{n, p}(x_2) = 
(n-1)n(n+p-1){x_2}^4 -4 (n-1)\left(2 n^2-2 n-p^2+p\right){x_2}^3  \nonumber  \\
& & + 2 \left(12 n^3-11 n^2 p-25 n^2-2 n p^2+ 20 n p +14 n+2 p^3-2 p^2-6 p-2\right){x_2}^2  \nonumber \\
   & &  - 8  (n-1)(4 n-3 p-1) (n-p-1){x_2} + 8(n-p-1)^2(2 n-p-1) = 0. \label{4}
     \end{eqnarray}
From (\ref{2}) we get that 
  \begin{eqnarray}  x_2 = - \frac{({x_4}-2) ({x_4} (2 n-p-1)- 2 (p-1))}{{x_4}
   (n-p-1)}.  \label{5}
  \end{eqnarray} 
  Note that $x_2 > 0 $ if and only if $\displaystyle \frac{2(p-1)}{2n-p-1} < x_4 < 2$. 
By substituting   equation (\ref{5})  into equation (\ref{1}),  we obtain the following equation:  
 \begin{eqnarray}
  & & \ \    G_{n, p}(x_4) = 
(n-1) n (2 n-p-1){x_4}^4 -4 (n-1)\left(n^2+2 n p-n-p^2-p\right){x_4}^3  \nonumber  \\
& & +2\left(n^3+9 n^2 p-7 n^2+4 n p^2-16 n p+8
   n-2 p^3-2 p^2+6 p-2\right) {x_4}^2 \nonumber \\
   & & -8 (n-1) (p-1) {x_4} (n+3 p-1)+8 (p-1)^2(n+p-1) = 0. 
     \end{eqnarray}
  Note that  the relation between $H_{n, p}(x_2)$ and  $G_{n, p}(x_4)$ is given by 
    \begin{eqnarray} G_{n, p}(x_4) = H_{n, n-p}(x_4).  \label{7}
\end{eqnarray}

\begin{prop}\label{existence}
The equation $ H_{n, p}(x_2) = 0$ has at least two solutions between  $\displaystyle x_2 = \frac{2(n-p-1)}{n+p-1}$  and $ x_2 = 2$.
\end{prop}
\begin{proof}  We consider the value   $H_{n, p}(x_2)$ at $\displaystyle x_2 = \frac{2(n-p-1)}{n+p-1}$  and $ x_2 = 2$. We see that  
 $$\displaystyle H_{n, p}\left(\frac{2(n-p-1)}{n+p-1} \right) = \frac{8 (p-1)^3 (n-p-1)^2}{(n+p-1)^2} > 0  \quad 
 \mbox{and } \quad 
  H_{n, p}\left(2 \right) = 8 (p-1)^3 > 0. $$
Now,  the value   $H_{n, p}(x_2)$ at $\displaystyle x_2 = \frac{2(n-p-1)}{n}$ is given by
 $$\displaystyle H_{n, p}\left(\frac{2(n-p-1)}{n} \right) = -\frac{16 (p-1)^2 (n-p-1)^3}{n^3} < 0, 
 $$ 
 thus the equation $ H_{n, p}(x_2) = 0$ has at least two solutions between  $\displaystyle x_2 = \frac{2(n-p-1)}{n+p-1}$  and $ x_2 = 2$.      
\end{proof}

We need to show that the  polynomial $ H_{n, p}(x_2)$ has only one local minimum (i.e. the two solutions obtained in Proposition
\ref{existence} are unique), with some exceptions which will also be studied.

\begin{lemma}\label{2p+5}
For $n\ge 2p+5$ and $p\ge 4$ the equation $ H_{n, p}(x_2)=0$ has  exactly two positive solutions.
\end{lemma}
\begin{proof}  We have that
  \begin{eqnarray*}
  & & \ \    \frac{d H_{n, p}}{d x_2} = 4 (n-1) n (n+p-1){x_2}^3-12 (n-1) \left(2 n^2-2 n-p^2+p\right){x_2}^2 \\
  & & +4
   \left(12 n^3-11 n^2 p-25 n^2-2 n p^2+20
   n p+14 n+2 p^3-2 p^2-6 p-2\right){x_2}\\
   & &  -8 (n-1) (4 n-3 p-1) (n-p-1),
     \end{eqnarray*}
    \begin{eqnarray*}  
  & & \ \    \frac{d^2 H_{n, p}}{d{ x_2}^2} = 12 (n-1) n (n+p-1){x_2}^2 - 24 (n-1) \left(2 n^2-2 n-p^2+p\right){x_2} \\
  & & +4
   \left(12 n^3-11 n^2 p-25 n^2-2 n p^2+20
   n p+14 n+2 p^3-2 p^2-6 p-2\right) 
   \end{eqnarray*}
 and 
    \begin{eqnarray*}  
  & & \ \    \frac{d^3 H_{n, p}}{d{ x_2}^3} = 24 (n-1) n (n+p-1){x_2} - 24 (n-1) \left(2 n^2-2 n-p^2+p\right). 
   \end{eqnarray*}
Note that the quadratic polynomial  $\displaystyle \frac{d^2 H_{n, p}}{d{ x_2}^2} $  attains its minimum at $\displaystyle x_2 = \frac{2 n^2-2 n-p^2+p}{n (n+p-1)}$ and we see that  
     \begin{eqnarray*}  
  & & \ \  \frac{d^2 H_{n, p}}{d{ x_2}^2}\left( \frac{2 n^2-2 n-p^2+p}{n (n+p-1)} \right)  = 
  \frac{4}{n (n+p-1)} \left(n^4 p-n^4-n^3 p^2-6 n^3 p+3 n^3 \right.\\
  & &  \left. -4 n^2 p^2+12 n^2 p-4 n^2-n p^4+2 n p^3+5 n p^2-8 n p+2
   n+3 p^4-6 p^3+3 p^2\right).    
     \end{eqnarray*}  
     
We set \begin{eqnarray*}   & & \ \ M(n, p) = n^4 p-n^4-n^3 p^2-6 n^3 p+3 n^3 -4 n^2 p^2+12 n^2 p-4 n^2-n p^4+2 n p^3 \\  
& & +5 n p^2 -8 n p+2 n+3 p^4-6 p^3+3 p^2 
 \end{eqnarray*}  
 and  we investigate the conditions for $n, p$ such  that $ M(n, p) > 0 $ for $n \geq 2 p$. 

We consider the coefficients of $ M(n, p)$ as a polynomial of $ n -2 p- 5$. 
We can write  $ M(n, p) $ as 
 \begin{eqnarray*}   &  &  \quad   M(n, p) = (p-1) (n-2 p-5)^4+ 
 \left(7 p^2+6 p-17\right) (n-2 p-5)^3 \\
 & & +\left(18 p^3+41 p^2-30 p-109\right) (n-2 p-5)^2 
   \\
   & & +\left(19
   p^4+62 p^3-26 p^2-274 p-313\right) (n-2
   p-5) \\
   & & +6 p^5+22 p^4-64 p^3-309
   p^2-491 p-340. 
    \end{eqnarray*}
  We put 
   $$ \begin{array}{lcl}    a_0 = 6 p^5+22 p^4-64 p^3-309 p^2-491 p-340, 
 & &  a_1 = 19 p^4+62 p^3-26 p^2-274 p-313, 
 \\
   a_2 = 18 p^3+41 p^2-30 p-109, 
&  &  a_3 = 7 p^2+6 p-17.  
  \end{array}$$
Note that 
 \begin{eqnarray*} & &   a_0   = 6
   (p-4)^5+142 (p-4)^4+1248 (p-4)^3+4875
   (p-4)^2+7277 (p-4)+432, \\
   & & a_1 =19 (p-3)^4+290 (p-3)^3+1558 (p-3)^2+3296 (p-3)+1844, 
   \\
   & &  a_2 =  18 (p-2)^3+149 (p-2)^2+350
   (p-2)+139, 
    \\
   & &
 a_3 =  7 (p-2)^2+34 (p-2)+23. 
    \end{eqnarray*}  
    Thus we see that $a_0 > 0 $, 
     $ a_1 >  0, a_2 > 0, a_3 > 0 $ for $p \geq 4$. Therefore we see that 
     $\displaystyle \frac{d^2 H_{n, p}}{d{ x_2}^2} > 0 $  for $ n \geq 2 p+5$ and $p \geq 4$ and hence,  $\displaystyle \frac{d H_{n, p}}{d{ x_2}}(x_2) $ is monotone increasing and the  polynomial $ H_{n, p}(x_2)$ has only one local minimum for  $ n \geq 2 p +5$ and $p \geq4$. Thus the equation $ H_{n, p}(x_2) = 0$ has exactly two positive solutions. 
\end{proof}

Now we examine the values $p=2$ and $p=3$.

\begin{lemma}\label{p=2_3}  {\rm{(1)}} \ Let $p=2$.  Then for $n\ge 7$ the equation $ H_{n, 2}(x_2) = 0$ has exactly two positive solutions,
and for  $4\le n\le 6$ it 
has exactly four positive solutions.

\noindent
{\rm{(2)}} \ Let $p=3$.  Then for $n\ge 8$  the equation $ H_{n, 3}(x_2) = 0$ has exactly two positive solutions, and for 
$6\le n\le 7$ it 
has exactly four positive solutions.    
\end{lemma}
\begin{proof}  (1)\ For  $p = 2$ we have that
       \begin{eqnarray*} 
 \ \   & &   M(n, 2) = n^4-13 n^3+4 n^2+6 n+12 \\
 & & =  (n-13)^4+39 (n-13)^3+511 (n-13)^2+2307 (n-13)+766. 
     \end{eqnarray*}  
    Thus we see that  
    $\displaystyle \frac{d^2 H_{n, 2}}{d{ x_2}^2} > 0 $  for $ n \geq 13$,  and hence,  $\displaystyle \frac{d H_{n, 2}}{d{ x_2}}(x_2) $ is monotone increasing and the  polynomial $ H_{n, 2}(x_2)$ has only one local minimum for  $ n \geq 13$. Thus the equation $ H_{n, 2}(x_2) = 0$ has exactly two positive solutions for $ n \geq 13$. 
    For $ 4 \leq n \leq 12$, we consider   polynomials $ H_{n, 2}(x_2)$ one by one and we see that, for $ 7 \leq n \leq 12$ the equation $ H_{n, 2}(x_2) = 0$ has  two positive solutions, and for $ 4 \leq n \leq 6$ the equation $ H_{n, 2}(x_2) = 0$ has four positive solutions.

   \begin{figure}[htbp]
\begin{minipage}{0.32\linewidth}
\begin{center}
\includegraphics[width=0.99\linewidth]{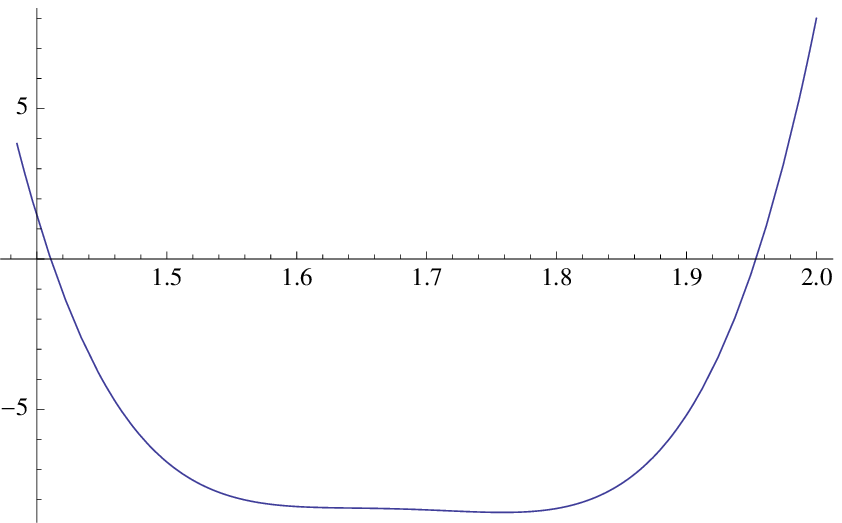}

\footnotesize $H_{12, 2}(x_2)$
\end{center}
\end{minipage}
\begin{minipage}{0.32\linewidth}
\begin{center}
\includegraphics[width=0.99\linewidth]{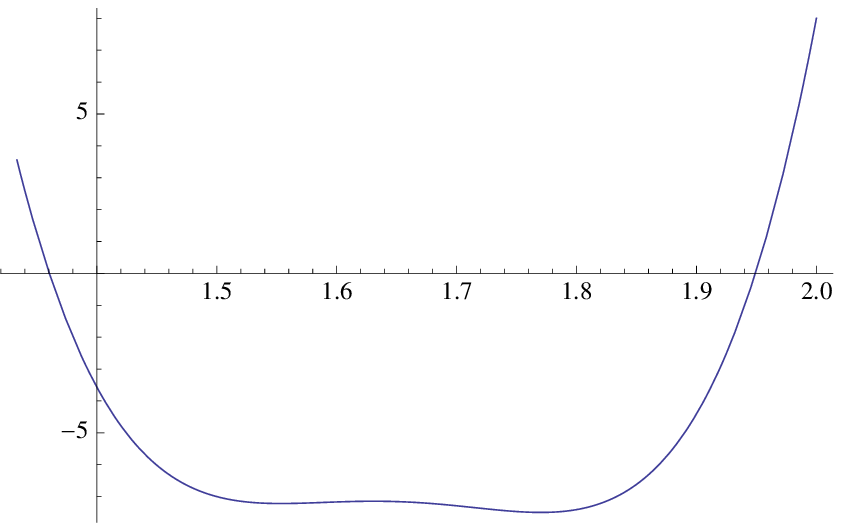}

\footnotesize $H_{11, 2}(x_2)$
\end{center}
\end{minipage}
\begin{minipage}{0.32\linewidth}
\begin{center}
\includegraphics[width=0.99\linewidth]{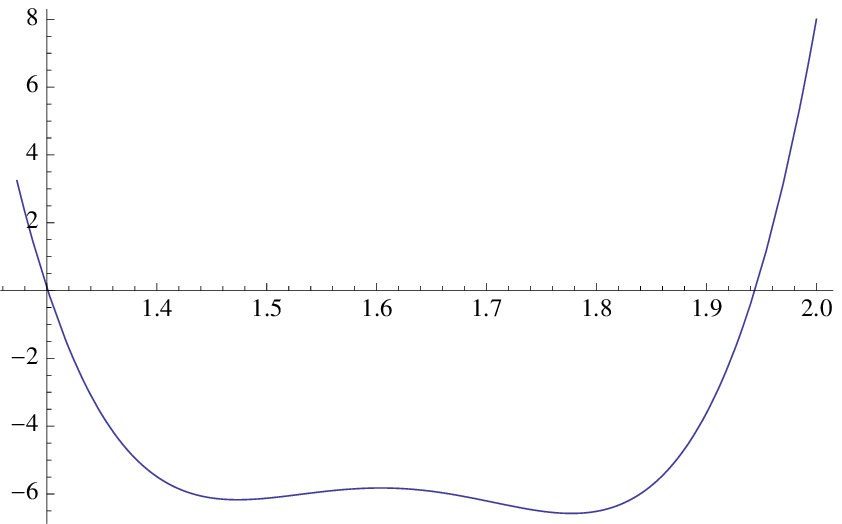}

\footnotesize $H_{10, 2}(x_2)$
\end{center}
\end{minipage}
\end{figure}
   \begin{figure}[htbp]
\begin{minipage}{0.32\linewidth}
\begin{center}
\includegraphics[width=0.99\linewidth]{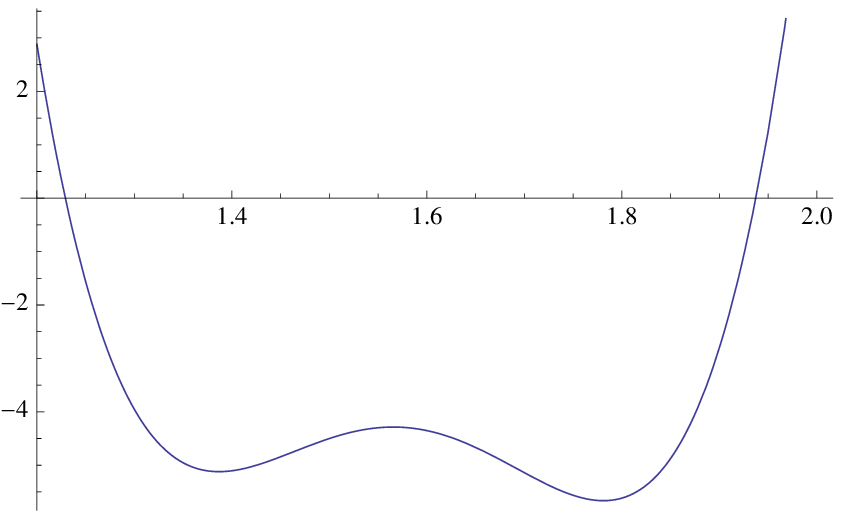}

\footnotesize $H_{9, 2}(x_2)$
\end{center}
\end{minipage}
\begin{minipage}{0.32\linewidth}
\begin{center}
\includegraphics[width=0.99\linewidth]{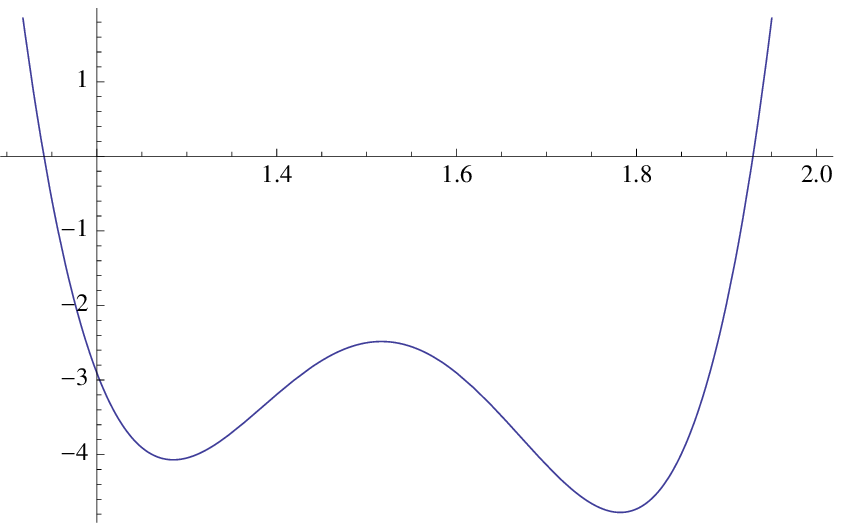}

\footnotesize $H_{8, 2}(x_2)$
\end{center}
\end{minipage}
\begin{minipage}{0.32\linewidth}
\begin{center}
\includegraphics[width=0.99\linewidth]{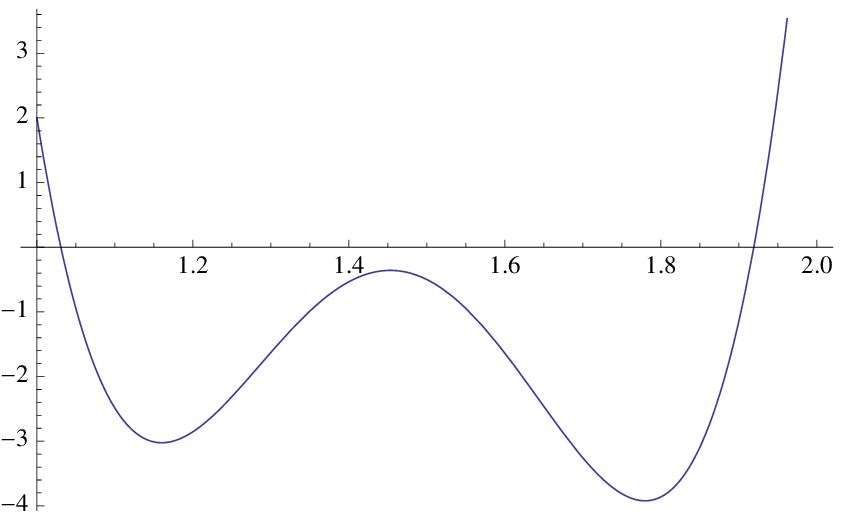}

\footnotesize $H_{7, 2}(x_2)$
\end{center}
\end{minipage}
\end{figure}
   \begin{figure}[htbp]
\begin{minipage}{0.32\linewidth}
\begin{center}
\includegraphics[width=0.99\linewidth]{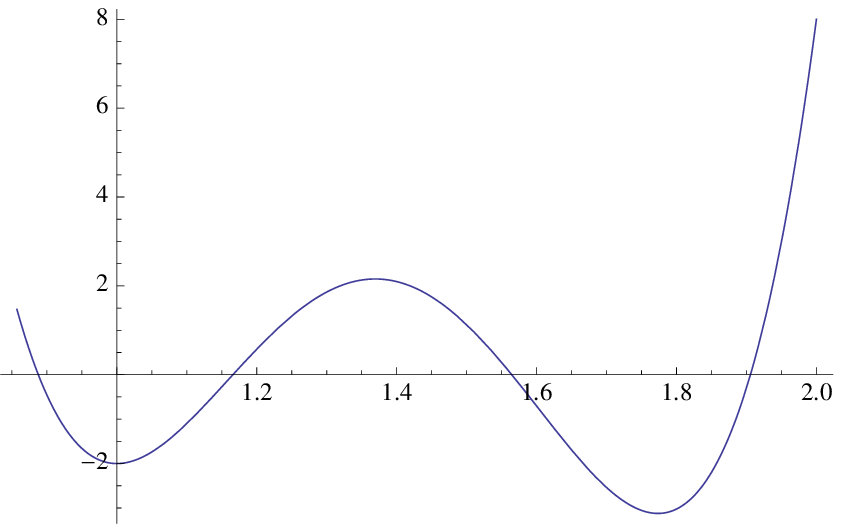}

\footnotesize $H_{6, 2}(x_2)$
\end{center}
\end{minipage}
\begin{minipage}{0.32\linewidth}
\begin{center}
\includegraphics[width=0.99\linewidth]{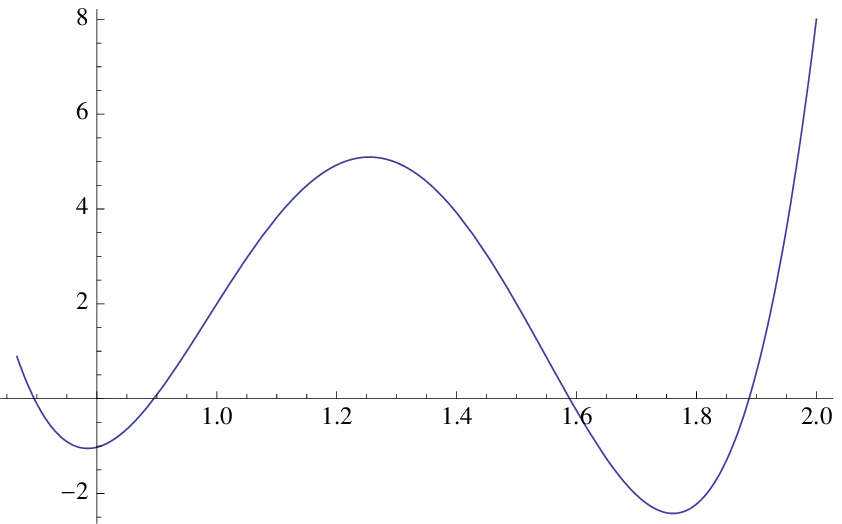}

\footnotesize $H_{5, 2}(x_2)$
\end{center}
\end{minipage}
\begin{minipage}{0.32\linewidth}
\begin{center}

\end{center}
\end{minipage}
\end{figure}

(2)\   For  $p =3$ we have that
       \begin{eqnarray*} 
 \ \   & &   M(n, 3) = 2 \left(n^4-12 n^3-2 n^2-2 n+54\right) \\
 & & =  (n-13)^4+40 (n-13)^3+544 (n-13)^2+2650 (n-13)+1887. 
     \end{eqnarray*}  
    Thus we see that  
    $\displaystyle \frac{d^2 H_{n, 3}}{d{ x_2}^2} > 0 $  for $ n \geq 13$,  and hence,  $\displaystyle \frac{d H_{n, 3}}{d{ x_2}}(x_2) $ is monotone increasing and the  polynomial $ H_{n, 3}(x_2)$ has only one local minimum for  $ n \geq 13$. Thus the equation $ H_{n, 3}(x_2) = 0$ has exactly two positive solutions for $ n \geq 13$. 
    For $ 6 \leq n \leq 12$, we consider   polynomials $ H_{n, 3}(x_2)$ one by one and we see that, for $ 8 \leq n \leq 12$ the equation $ H_{n, 3}(x_2) = 0$ has  two positive solutions, and for $ 6 \leq n \leq 7$ the equation $ H_{n, 3}(x_2) = 0$ has four positive solutions.  
    
\begin{minipage}{.32\linewidth}
\begin{center}
\includegraphics[width=0.95\linewidth]{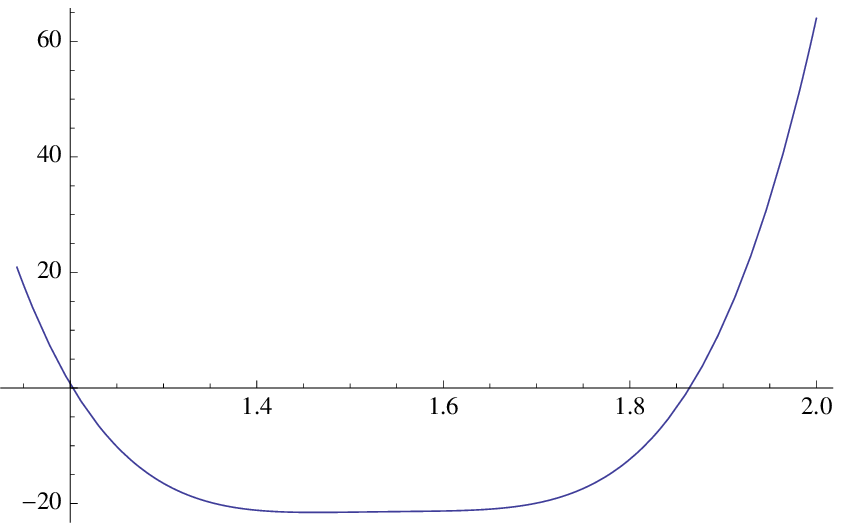}

\footnotesize $H_{12, 3}(x_2)$
\end{center}
\end{minipage}
\begin{minipage}{0.32\linewidth}
\begin{center}
\includegraphics[width=0.95\linewidth]{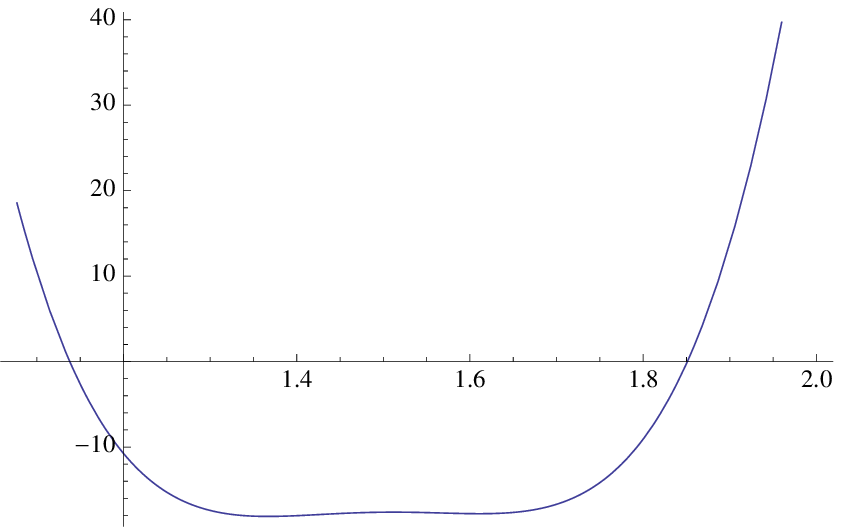}

\footnotesize $H_{11, 3}(x_2)$
\end{center}
\end{minipage}
\begin{minipage}{0.32\linewidth}
\begin{center}
\includegraphics[width=0.95\linewidth]{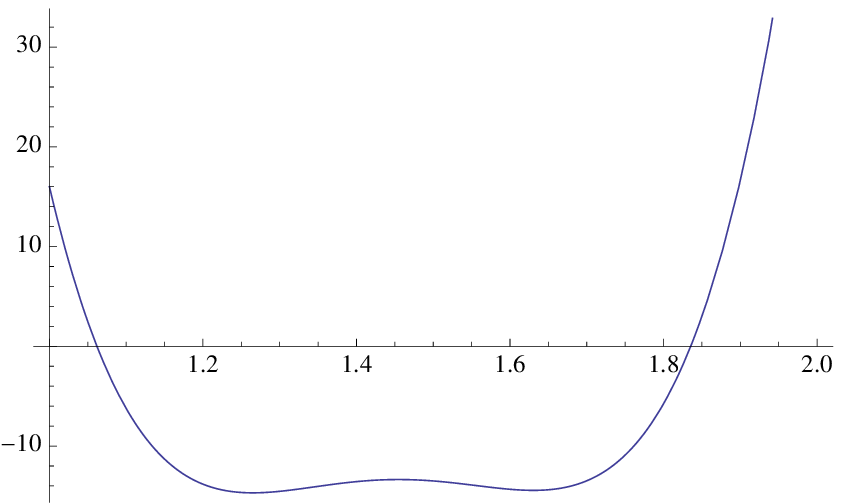}

\footnotesize $H_{10, 3}(x_2)$
\end{center}
\end{minipage}

\begin{minipage}{.32\linewidth}
\begin{center}
\includegraphics[width=0.95\linewidth]{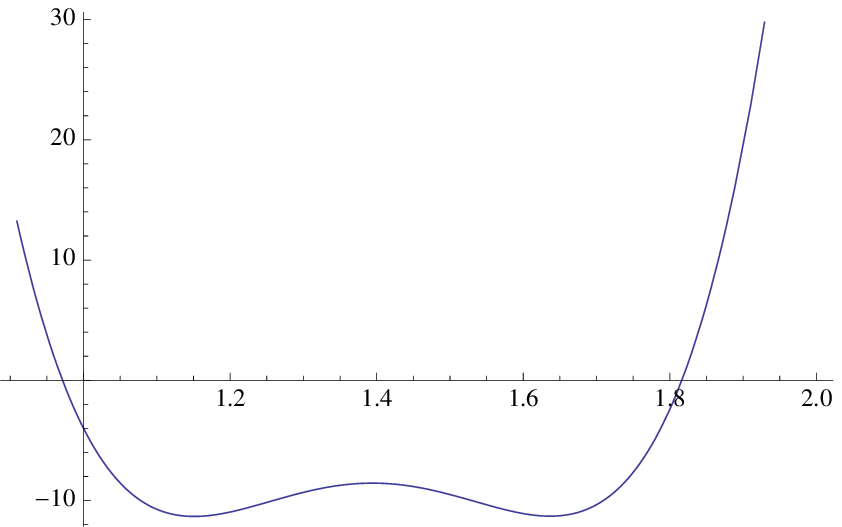}

\footnotesize $H_{9, 3}(x_2)$
\end{center}
\end{minipage}
\begin{minipage}{0.32\linewidth}
\begin{center}
\includegraphics[width=0.95\linewidth]{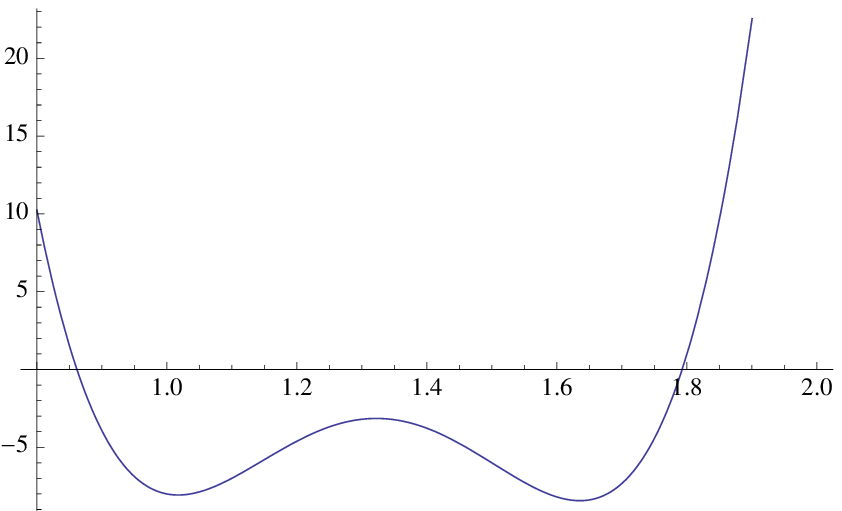}

\footnotesize $H_{8, 3}(x_2)$
\end{center}
\end{minipage}
\begin{minipage}{0.32\linewidth}
\begin{center}
\includegraphics[width=0.95\linewidth]{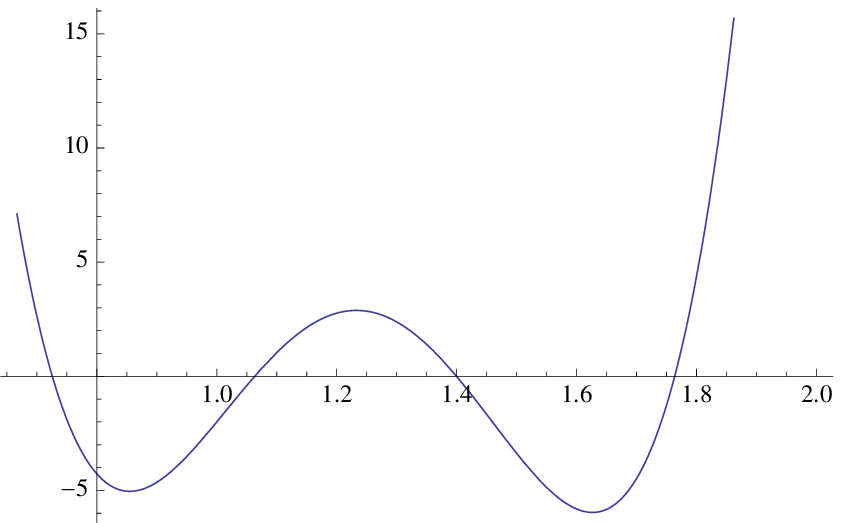}

\footnotesize $H_{7, 3}(x_2)$
\end{center}
\end{minipage}

\begin{minipage}{.32\linewidth}
\begin{center}
\includegraphics[width=0.95\linewidth]{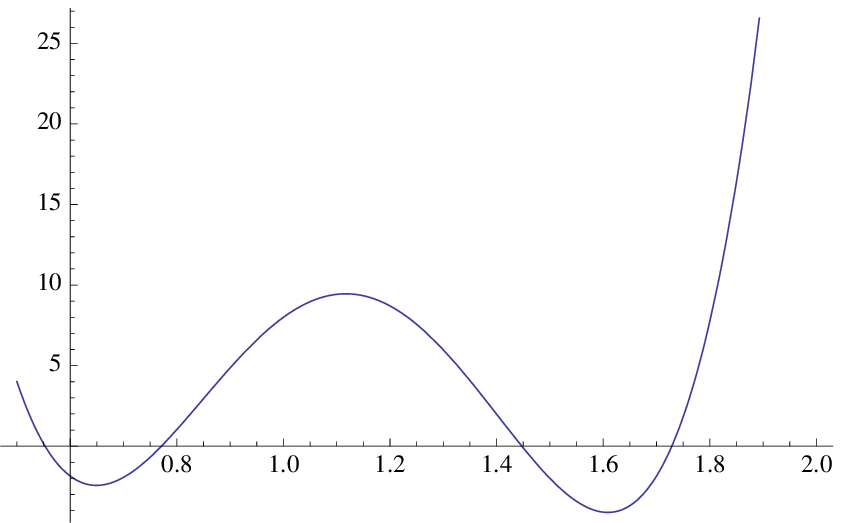}

\footnotesize $H_{6, 3}(x_2)$
\end{center}
\end{minipage}
\end{proof}

Next, we consider the case when $2p\le n\le 2p+4$.  We may assume that $p\ge 4$.

\begin{lemma}\label{n=2p}  Let $n=2p$.  Then the equation $ H_{2 p, p}(x_2) = 0$ has exactly two positive solutions  for $ p \geq 7$ and    
four positive solutions for $4 \leq p\leq 6$.  
\end{lemma}
\begin{proof}  We see that 
    \begin{eqnarray*} 
    & &  \ \   H_{2 p, p}(x_2) = 2 \left((2 p-1) {x_2}^2-2 (2 p-1) {x_2}+2
   (p-1)\right)\times \\
 & &    \left(p (3 p-1) {x_2}^2-4 p (2
   p-1) {x_2}+2 (p-1) (3 p-1)\right)
     \end{eqnarray*} 
     Thus,   the four solutions  of the equation   $ H_{2 p, p}(x_2) = 0$ are given by 
    \begin{equation}\label{400}
    (a) \ x_2 = \frac{2 p\pm\sqrt{2 p-1}-1}{2 p-1},  \quad  (b) \
    x_2 = \frac{2 p (2 p-1) \pm \sqrt{2} \sqrt{-p \left(p^3-7 p^2+5 p-1\right)}}{p (3 p-1)}.
   \end{equation}
 Since   $-p \left(p^3-7 p^2+5 p-1\right)$ is negative for $ p \geq 7$,  we  see that the equation $ H_{2 p, p}(x_2) = 0$ has exactly two positive solutions  $ p \geq 7$ and    four positive solutions for $4 \leq p\leq 6$.  
    \begin{figure}[htbp]
\begin{minipage}{.32\linewidth}
\begin{center}
\includegraphics[width=0.95\linewidth]{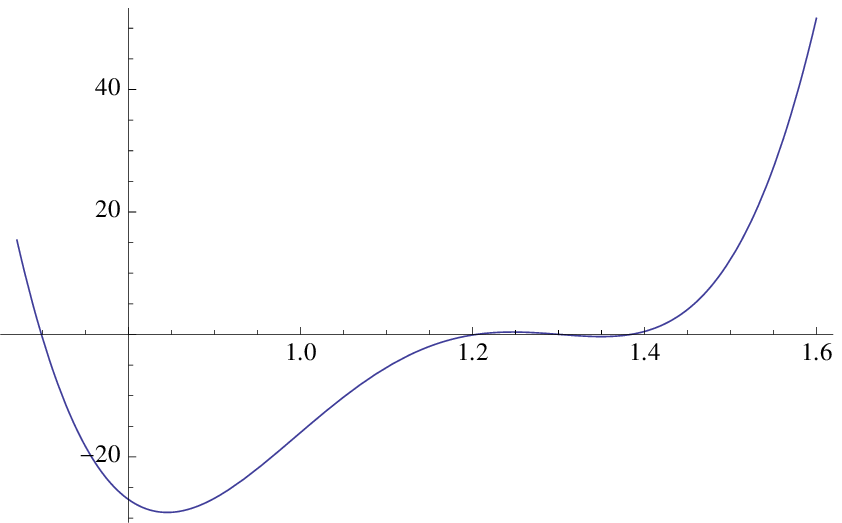}

\footnotesize $H_{12, 6}(x_2)$
\end{center}
\end{minipage}
\begin{minipage}{0.32\linewidth}
\begin{center}
\includegraphics[width=0.95\linewidth]{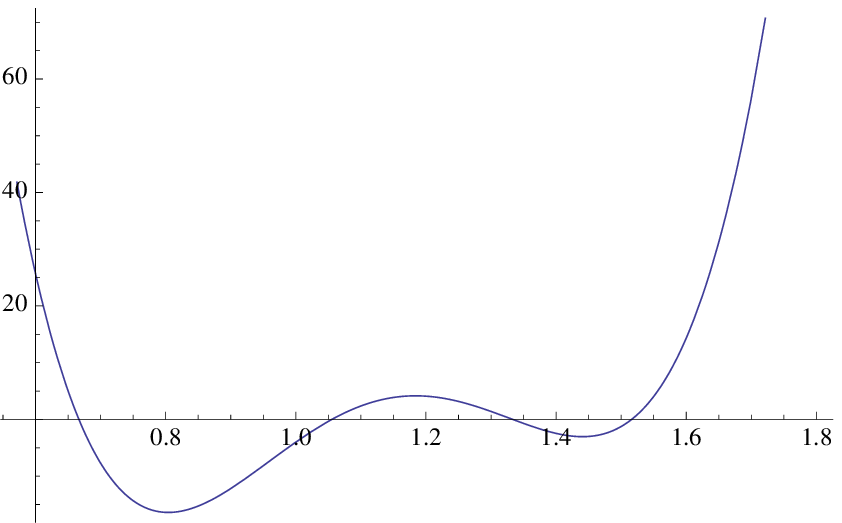}

\footnotesize $H_{10, 5}(x_2)$
\end{center}
\end{minipage}
\begin{minipage}{0.32\linewidth}
\begin{center}
\includegraphics[width=0.95\linewidth]{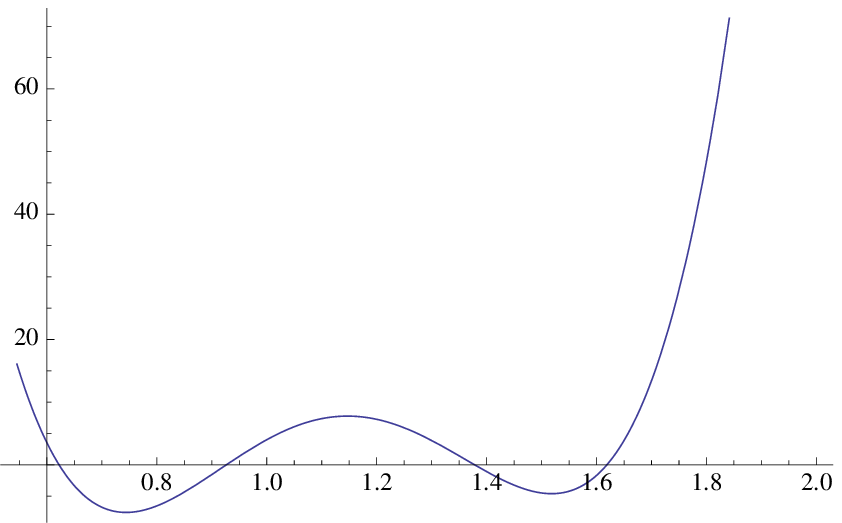}

\footnotesize $H_{8, 4}(x_2)$
\end{center}
\end{minipage}
\end{figure}

\end{proof}

\begin{lemma}\label{n=2p+1} Let $n=2p+1$.  Then for $p\ge 4$ the equation $ H_{2 p+1, p}(x_2) = 0$ has exactly two positive solutions.  
\end{lemma}
\begin{proof}  We see that 
     \begin{eqnarray*}  
  & & \ \   H _{2 p+1,  p}(x_2) = 6 p^2 (2 p+1) {x_2}^4-8 p^2 (7 p+5)
   {x_2}^3+2 \left(50 p^3+36 p^2+3 p-1\right)
   {x_2}^2 \\
   & & -16 p^2 (5 p+3) {x_2}+8 p^2 (3
   p+1)
     \end{eqnarray*}
      and 
      \begin{eqnarray*}  
      & &  \frac{d^2 H_{2 p +1, p}}{d{ x_2}^2}\left( \frac{2 n^2-2 n-p^2+p}{n (n+p-1)} \right) = \frac{4 \left(2 p^4-18 p^3-8 p^2+p-1\right)}{2 p+1}\\
      & & 
=  \frac{4 \left(2 (p-9)^4+54 (p-9)^3+478 (p-9)^2+1315 (p-9)-640\right)}{2 p+1}. 
 \end{eqnarray*}
  Thus we see that 
     $\displaystyle \frac{d^2 H_{2 p +1, p}}{d{ x_2}^2} > 0 $  for  $p \geq 10$ and hence,  $\displaystyle \frac{d H_{2 p+1, p}}{d{ x_2}}(x_2) $ is monotone increasing and the  polynomial $ H_{2 p+1, p}(x_2)$ has only one local minimum for   $p \geq 10$. Thus the equation $ H_{2 p +1, p}(x_2) = 0$ has exactly two positive solutions. 

 For $4 \leq p \leq 9$,   we see that $\displaystyle \frac{d^2 H_{2 p+1, p}}{d{ x_2}^2}\left( \frac{2 n^2-2 n-p^2+p}{n (n+p-1)} \right)$ is negative and two real solutions $\alpha, \beta$ of the quadratic equation $\displaystyle \frac{d^2 H_{2 p +1, p}}{d{ x_2}^2}  = 0 $ are given by 
   \begin{eqnarray*} 
   &  &  \alpha = \frac{2 p^2 (7
   p+5)-\sqrt{2} \sqrt{-p^2 \left(2 p^4-18 p^3-8
   p^2+p-1\right)}}{6 \left(2
   p^3+p^2\right)}, \\
    &  & \beta = \frac{2 p^2 (7
   p+5) + \sqrt{2} \sqrt{-p^2 \left(2 p^4-18 p^3-8
   p^2+p-1\right)}}{6 \left(2
   p^3+p^2\right)}. 
      \end{eqnarray*}  
 Since the polynomial $\displaystyle \frac{d H_{2 p +1, p}}{d{ x_2}}(x_2) $ of degree 3 takes a local minimum at  $\displaystyle x_2 = \beta$, we consider the value $\displaystyle \frac{d H_{2 p +1, p}}{d{ x_2}}(\beta) $. We see that 
    \begin{eqnarray*} 
& & \ \ \  \frac{d H_{2 p +1, p}}{d{ x_2}}(\beta) = 
\frac{2}{9 p^4 (2 p+1)^2} \left(2 (p-1)^2 \left(8 p^3-14 p^2-36
   p-15\right) p^4 \right. \\
    & &  \left. +2 \sqrt{2} \left(2 p^4-18 p^3-8
   p^2+p-1\right) \sqrt{-p^2 \left(2 p^4-18 p^3-8
   p^2+p-1\right)} p^2\right). 
 \end{eqnarray*}  
       By evaluating the above expression for the integers $4 \leq p \leq 9$, 
       we see that  
$\displaystyle \frac{d H_{2 p+1, p}}{d{ x_2}}(\beta)  > 0 $ for $ 6 \leq p \leq 9$ and $\displaystyle \frac{d H_{2 p+1, p}}{d{ x_2}}(\beta)  <  0 $ for $ 4 \leq p \leq 5$.  Thus the  polynomial $ H_{2 p+1, p}(x_2)$ has only one local minimum for  $6 \leq p \leq 9$,  and        
  $ H_{2 p +1, p}(x_2)$ has two local minima  and one local maximum  for  $4 \leq p \leq 5$.   However, we see that 
  for $ p = 4, 5$ the equation $ H_{2 p+1, p}(x_2) = 0$ has exactly two  roots  between $\displaystyle  \frac{2(n-p-1)}{(n+p-1)} = \frac{2}{3}$ and $2$, and this completes the proof.  
    \begin{figure}[htbp]
    \begin{center}
\begin{minipage}{.32\linewidth}
\begin{center}
\includegraphics[width=0.95\linewidth]{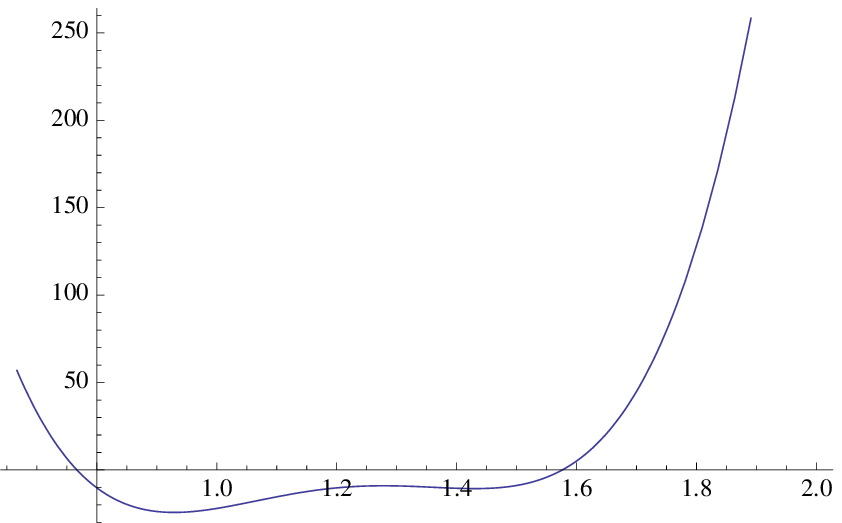}

\footnotesize $H_{11, 5}(x_2)$
\end{center}
\end{minipage}
\begin{minipage}{0.32\linewidth}
\begin{center}
\includegraphics[width=0.95\linewidth]{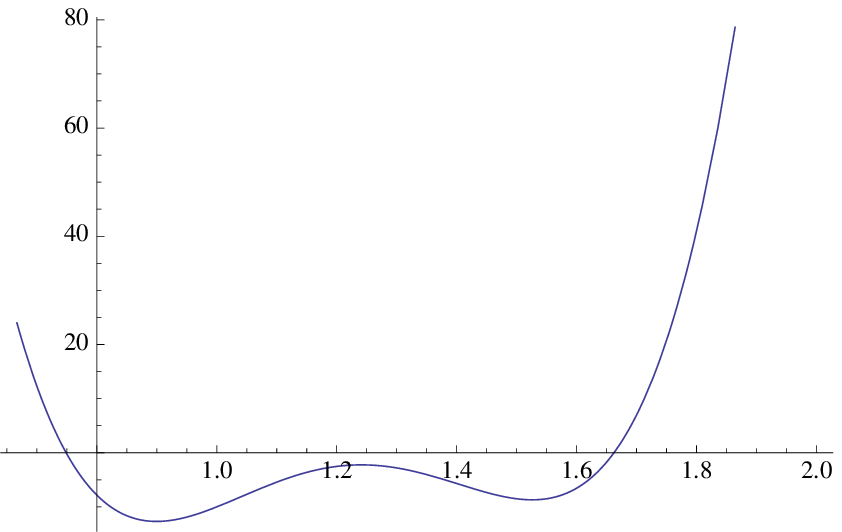}

\footnotesize $H_{9, 4}(x_2)$
\end{center}
\end{minipage}
\end{center}
\end{figure}
 
\end{proof}

\begin{lemma}\label{n=2p+2} Let $n=2p+2$.  Then for $p\ge 4$ the equation $ H_{2 p+2, p}(x_2) = 0$ has exactly two positive solutions.  
\end{lemma}
\begin{proof} We see that 
     \begin{eqnarray*}  
  & & \ \   H _{2 p+2,  p}(x_2) = (2 p+1) (2 p+2) (3 p+1){x_2}^4-4 (2 p+1) \left(7 p^2+13 p+4\right) {x_2}^3 \\
  & & 
+4(p+1) \left(25 p^2+42 p+11\right){x_2}^2-8 (p+1) (2 p+1) (5 p+7)
   {x_2}+24 (p+1)^3 
        \end{eqnarray*}
      and 
      \begin{eqnarray*}  
      & &  \frac{d^2 H_{2 p +2, p}}{d{ x_2}^2}\left( \frac{2 n^2-2 n-p^2+p}{n (n+p-1)} \right) = \frac{2 ( 6 p^5-35 p^4-88 p^3-51 p^2-20 p-4)}{(3 p+1)(p+1)}\\
      & & 
=  \frac{2 \left(6 (p-8)^5+205 (p-8)^4+2632 (p-8)^3+15117
   (p-8)^2+33468 (p-8)+4764\right)}{(3 p+1)(p+1)}. 
 \end{eqnarray*}
  Thus we see that 
     $\displaystyle \frac{d^2 H_{2 p +2, p}}{d{ x_2}^2} > 0 $  for  $p \geq 8$ and hence,  $\displaystyle \frac{d H_{2 p+2, p}}{d{ x_2}}(x_2) $ is monotone increasing and the  polynomial $ H_{2 p+2, p}(x_2)$ has only one local minimum for   $p \geq 8$. Thus the equation $ H_{2 p+2, p}(x_2) = 0$ has exactly two positive solutions. 

For $4 \leq p \leq 7$,   we see that $\displaystyle \frac{d^2 H_{2 p+2, p}}{d{ x_2}^2}\left( \frac{2 n^2-2 n-p^2+p}{n (n+p-1)} \right)$ is negative and the two real solutions $\alpha, \beta$ of the quadratic equation $\displaystyle \frac{d^2 H_{2 p +2, p}}{d{ x_2}^2}  = 0 $ are given by 
   \begin{eqnarray*} 
   &  &  \alpha = \frac {3 (2 p + 1) \left(7 p^2 + 13 p +  4 \right) - \sqrt{3} \sqrt{(-2 p - 1) \left(6 p^5 -  35 p^4 - 88 p^3 - 51 p^2 - 20 p - 4 \right)}} {6 (p + 1) (2 p + 1) (3 p + 1)}, \\
    &  & \beta = \frac {3 (2 p + 1) \left(7 p^2 + 13 p + 4 \right) + \sqrt{3} \sqrt{(-2 p - 1) \left(6 p^5 -  35 p^4 - 88 p^3 - 51 p^2 - 20 p - 4 \right)}} {6 (p + 1) (2 p + 1) (3 p + 1)}. 
      \end{eqnarray*}  
 Since the polynomial $\displaystyle \frac{d H_{2 p +2, p}}{d{ x_2}}(x_2) $ of degree 3 takes local minimum at  $\displaystyle x_2 = \beta$, we consider the value $\displaystyle \frac{d H_{2 p +2, p}}{d{ x_2}}(\beta) $. We see that 
    \begin{eqnarray*} 
& & \ \ \  \frac{d H_{2 p +2, p}}{d{ x_2}}(\beta) = 
\frac{1}{9 (p+1)^2 (2 p+1) (3 p+1)^2} \left(18 (2 p+1) \left(4 p^2+7 p+2\right) \right. \times \\
&  &
   \left(p^3-p^2-6 p-2\right) (p-1)^2 
  + 2 \sqrt{3}
   \left(6 p^5-35 p^4-88 p^3-51 p^2-20 p-4\right)\times \\
&  & \left.  \sqrt{(-2 p-1) \left(6 p^5-35 p^4-88 p^3-51
   p^2-20 p-4\right)}\right). 
 \end{eqnarray*}  
       By substituting integer $4 \leq p \leq 7$, we see that  $\displaystyle \frac{d H_{2 p+2, p}}{d{ x_2}}(\beta)  > 0 $ for $ 5 \leq p \leq 7$ and $\displaystyle \frac{d H_{2 p+2, p}}{d{ x_2}}(\beta)  <  0 $ for $ p= 4$.  Thus the  polynomial $ H_{2 p+2, p}(x_2)$ has only one local minimum for  $5 \leq p \leq 7$,  and        
  $ H_{2 p +2, p}(x_2)$ has two local minima  and one local maximum  for  $ p = 4$.   However, we see that 
  for $ p = 4$ the equation $ H_{2 p+2, p}(x_2) = 0$ has exactly two  roots  between $\displaystyle  \frac{2(n-p-1)}{(n+p-1)} = \frac{2(p +1)}{3p +1}$ and $2$.  
 \begin{figure}[htbp]
    \begin{center}
\begin{minipage}{.32\linewidth}
\begin{center}
\includegraphics[width=0.95\linewidth]{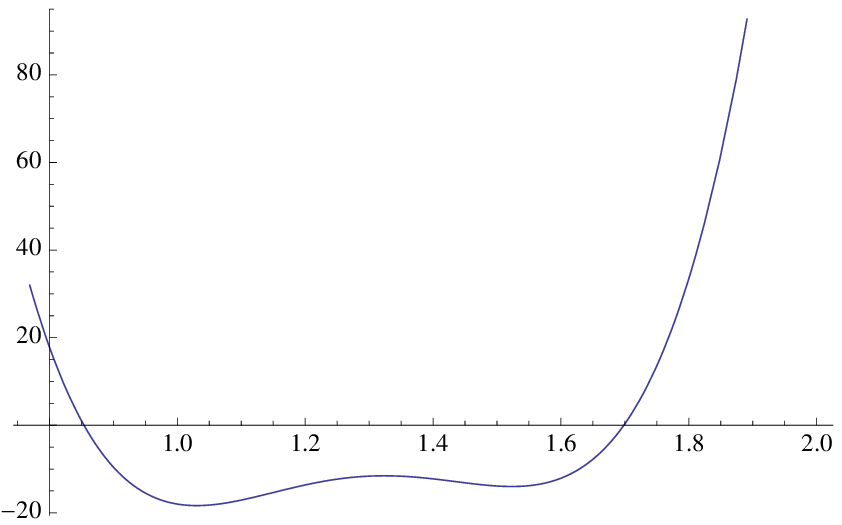}

\footnotesize $H_{10, 4}(x_2)$
\end{center}
\end{minipage}
\end{center}
\end{figure}

\end{proof}  

By a similar method we can prove the next two lemmas.

\begin{lemma}\label{n=2p+3} Let $n=2p+3$.  Then for $p\ge 4$ the equation $ H_{2 p+3, p}(x_2) = 0$ has exactly two positive solutions.  
\end{lemma}

\begin{lemma}\label{n=2p+4} Let $n=2p+4$.  Then for $p\ge 4$ the equation $ H_{2 p+4, p}(x_2) = 0$ has exactly two positive solutions.  
\end{lemma}

Therefore we have obtained the following:

\begin{prop}\label{solutions}
{\rm{(1)}}\ If $x_1 = x_3$ and $n \geq 2p$, then $M$ admits exactly four  $SO(2n )$-invariant Einstein metrics for the pairs $ (n, p) = (12, 6)$, $ (10, 5)$, $ (8, 4)$, $ (7, 3)$, $ (6, 3)$,  $(6, 2)$, $(5, 2)$, $ (4, 2)$ and two  $SO(2n )$-invariant Einstein metrics for all other cases.  

\noindent
{\rm{(2)}}\  If $x_1 = x_3$ and $ n \leq 2 p$, then $M$ admits exactly four  $SO(2n )$-invariant Einstein metrics for the pairs $ (n, p) = (12, 6)$, $ (10, 5)$, $ (8, 4)$, $ (7, 4)$, $(6, 4)$, $ (6, 3)$,  $(5, 3)$,  $ (4, 2)$ and two  $SO(2n )$-invariant Einstein metrics for all other cases.
\end{prop}
\begin{proof} Part (1) is a consequence of Proposition \ref{existence} and Lemmas \ref{2p+5} -- \ref{n=2p+4}.
For (2), we consider the equation $ G_{n, p}(x_4) = 0$,  and the result follows from the relation (\ref{7}).  
\end{proof}
  \noindent\underline{CASE B}
Let  
\begin{eqnarray}  
&  & \ \  - \,  x_1 x_2 + p x_1 x_2 - x_2 x_3 + p x_2 x_3 - x_1 x_4   + n x_1 x_4  -  p x_1 x_4  \nonumber  \\   
&  &  + \,   2  x_2  x_4  - 2 n x_2 x_4 - x_3 x_4 + n x_3 x_4 - p x_3 x_4  =0,   \label{31}
 \end{eqnarray}
 and set $x_1=1$. From equation (\ref{31}) we obtain that 
  \begin{eqnarray} 
  x_3 =  \frac{2 ( n-1) {x_2} {x_4}-(n - p - 1) {x_4}-(p- 1) {x_2}}{( n - p -1){x_4}+( p - 1){x_2}}.  \label{41}
 \end{eqnarray}
 We need to show the following:
  \begin{prop}\label{kahler}  The system of equations  
 $ r_1 - r_2 =0$ 
 and $ r_3 - r_4 =0$ has no positive solutions, except K\"ahler-Einstein metrics.  
 \end{prop}
 \begin{proof}  
  We  substitute equation (\ref{41}) and $x_1=1$ into the equations $ r_1 - r_2 =0$ and $ r_3 - r_4 = 0$, and we obtain the following equations :  
 {\footnotesize{\begin{eqnarray}
& &  F(x_2, x_4) = -(p-1) {x_2}^3 {x_4} \left(2 n^2-4 n+p^2+2
   p+1\right)-{x_2}^2 {x_4}^2 (3 n^3+5 n^2
   p-9 n^2  \nonumber \\
   & & -n p^2 -6 n p +7 n+p^3-3 p^2+3
   p-1)+(p-1)^2 {x_2}^4 (n+p-1)   \nonumber \\ 
  & &  +8 (n-1) (p-1)
   {x_2}^2 {x_4} (n+p-1) -4 (p-1)^2 {x_2}^2
   (n+p-1)  \nonumber \\ 
& &  +(p-1) {x_2} {x_4}^3 (n-p-1)^2   +8 (n-1) {x_2} {x_4}^2 (n-p-1) (n+p-1)   \nonumber \\
& &  -8 (p-1) {x_2} {x_4} (n-p-1) (n+p-1) -4 {x_4}^2  (n-p-1)^2 (n+p-1)  = 0,    \label{51}
 \\
 & & G(x_2, x_4) =   {x_2} {x_4}^3 (n-p-1) \left(3 n^2-2 n p-2
   n+p^2-2 p+1\right)   \nonumber  \\
& &  +{x_2}^2 {x_4}^2 (8
   n^3-6 n^2 p-18 n^2 
   +2 n p^2+12 n p+10 n-p^3-3 p^2-3 p-1) \nonumber  \\
 & &  -(p-1)^2 {x_2}^3 {x_4} (n-p-1) 
    -8 (n-1) (p-1) {x_2}^2 {x_4} (2 n-p-1)  \nonumber  \\
& & 
    +4 (p-1)^2 {x_2}^2 (2 n-p-1)-8 (n-1) {x_2}
   {x_4}^2 (n-p-1) (2 n-p-1)  \nonumber  \\
& &
   +8 (p-1) {x_2}
   {x_4} (n-p-1) (2 n-p-1)  -{x_4}^4
  (n-p-1)^2  (2 n-p-1)  \nonumber  \\
& & +4 {x_4}^2
   (n-p-1)^2 (2 n-p-1) =0.   \label{61}
    \end{eqnarray}}}
   
  We consider the resultant of the polynomials $ F(x_2, x_4)$ and  $ G(x_2, x_4)$ with respect to $x_2$,  which is a polynomial  of $x_4$, say 
  $Q(x_4)$.
  We factor $Q(x_4)$ as 
    \begin{eqnarray*} 
   & & 
Q(x_4) = 128 (n-1)^6 (p-1)^2 {x_4}^8 (n-p-1)^4 (n
   {x_4}-2 p+2) (n {x_4}-4 n+2 p+2)\times  \\
 & &   (3 n
   {x_4}-4 n-2 p {x_4}+2 p-2 {x_4}+2) (n
   {x_4}+2 p {x_4}-2 p-2 {x_4}+2)\times  \\
   & &  (6
   n^5 {x_4}^4+8 n^5 {x_4}^3+2 n^5
   {x_4}^2-3 n^4 p {x_4}^4-36 n^4 p
   {x_4}^3-38 n^4 p {x_4}^2-8 n^4 p
   {x_4}-17 n^4 {x_4}^4 \\
   & &
   -12 n^4 {x_4}^3+22
   n^4 {x_4}^2+8 n^4 {x_4}+72 n^3 p^2
   {x_4}^2+56 n^3 p^2 {x_4}+8 n^3 p^2+7 n^3 p
   {x_4}^4+116 n^3 p {x_4}^3  \\
   & & +36 n^3 p
   {x_4}^2-64 n^3 p {x_4}-16 n^3 p+15 n^3
   {x_4}^4-12 n^3 {x_4}^3-60 n^3 {x_4}^2+8
   n^3 {x_4}+8 n^3   \\
   & & +8 n^2 p^3 {x_4}^3+44 n^2 p^3
   {x_4}^2-48 n^2 p^3 {x_4}-24 n^2 p^3-24 n^2
   p^2 {x_4}^3-260 n^2 p^2 {x_4}^2-32 n^2 p^2
   {x_4}   \\
   & & +40 n^2 p^2-4 n^2 p {x_4}^4-104 n^2 p
   {x_4}^3+108 n^2 p {x_4}^2+112 n^2 p
   {x_4}-8 n^2 p-4 n^2 {x_4}^4+24 n^2
   {x_4}^3     \\
   & &  +44 n^2 {x_4}^2-32 n^2 {x_4}-8
   n^2-32 n p^4 {x_4}^2-80 n p^4 {x_4}-8 n p^3
   {x_4}^3-8 n p^3 {x_4}^2+256 n p^3
   {x_4}   \\
   & & +32 n p^3+24 n p^2 {x_4}^3+216 n p^2
   {x_4}^2-192 n p^2 {x_4}-64 n p^2+24 n p
   {x_4}^3-136 n p {x_4}^2+32 n p    \\
   & & -8 n
   {x_4}^3-8 n {x_4}^2+16 n {x_4}+32 p^5
   {x_4}+32 p^5+32 p^4 {x_4}^2-96 p^4-32 p^3
   {x_4}^2-128 p^3 {x_4}   \\
   & & +96 p^3-32 p^2
   {x_4}^2+128 p^2 {x_4}-32 p^2+32 p
   {x_4}^2-32 p {x_4}).   
   \end{eqnarray*} 
   
  We first consider the cases when 
    \begin{eqnarray*} 
 & & (n{x_4}-2( p-1)) (n {x_4}-2(2 n- p-1) \times \\
 & & \left((3 n - 2( p+1) ) {x_4}-2(2 n- p -1)\right) \left((n+2( p-1) ) {x_4}-2( p-1)\right) = 0, 
   \end{eqnarray*} 
  
  and we claim that we only get K\"ahler-Einstein metrics on $SO(2n )/ U(p) \times U( n -p)$. 
  
1)  Let $\displaystyle x_4 = \frac{2 (p-1)}{n}$. Then equations (\ref{51})  and  (\ref{61}) reduce to  
\begin{eqnarray*} 
  & &  \frac{(p-1)^2 (n \, {x_2} - 2( n+ p -1))}{n^3} 
  \left(n^2 (n+p-1) {x_2}^3 -2 (n-2) n (n-2 p) {x_2}^2   \right. \\
  & &  \left.   -4  (n-p-1) \left(n^2+2 n p - 4 n-p^2+1\right) {x_2}+8 n
   (n-p-1)^2  \right) =0,  \\
   & & \frac{2 (p-1)^2 (n-p-1)}{n^4} (n \,  {x_2} - 2 ( n + p -1) ) (n \,  {x_2} - 2 (n- p+1) )\times \\
   & &  \left(2 (n-p-1) (2 n-p-1)-n (p-1) {x_2}\right) = 0. 
      \end{eqnarray*} 
      If $ n \, {x_2}-2( n- p+1) \neq 0$, we have 
      \begin{eqnarray*} 
  & & 
\  \  \left(n^2 (n+p-1) {x_2}^3 -2 (n-2) n (n-2 p) {x_2}^2   \right. \\
  & &   \left.   -4  (n-p-1) \left(n^2+2 n p - 4 n-p^2+1\right) {x_2}+8 n
   (n-p-1)^2  \right) =0,  \\
   & & \  \ (n \,  {x_2} - 2 (n- p+1) )\left(2 (n-p-1) (2 n-p-1)-n (p-1) {x_2}\right) = 0. 
      \end{eqnarray*} 
 By taking   the resultant of these polynomials with respect to $x_2$, 
 we get 
 $$
 -2048 (n-1)^2 n^6 \left((n-p)^2+n-1\right) (n-p-1)^3
   (n-p),$$
  and we see that  the resultant is non-zero for $ 2 \leq p \leq n-2$. Thus we get only  $\displaystyle  {x_2} = \frac{2 ( n+ p-1)}{n} $ for a solution of equations (\ref{51})  and  (\ref{61}). From (\ref{41}), we see  $\displaystyle  {x_3} = \frac{  n+2 p- 2}{n} $. Thus we obtain a 
  K\"ahler-Einstein metric in this case.  
  
  Notice that this metric corresponds (up to scale) to the K\"ahler-Einstein metric $g_1$ of (\ref{KE}) 
   
2)  Let $\displaystyle x_4 = \frac{2 (2 n - p-1)}{n}$. Then equations  (\ref{51})  and  (\ref{61}) reduce to  
\begin{eqnarray*} 
  & & \ \  -\frac{(n\, {x_2}-2( n - p -1)) }{n^3}\left(-n^2
   (p-1)^2 (n+p-1) {x_2}^3  \right. \\
   & &  +2 n (p-1) 
   \left(4 n^3-3 n^2 p-9 n^2+2 n p^2+10 n p+4 n-4 p^2-4
   p\right){x_2}^2  \\
   & &  + 4 (2 n-p-1) \left(6 n^4+5 n^3
   p-19 n^3-13 n^2 p^2+21 n^2+5 n p^3+9 n p^2 \right. \\
   & & \left. \left. - 5 n p-9 n - p^4-2 p^3+2 p+1\right) {x_2} -8 n (n-p-1) (2 n-p-1)^2
   (n+p-1)\right)=0,  \\
   & & \ \  \frac{2 (2 n-p-1) (n\, {x_2}-2( n - p -1))}{n^4} \left(-n^2 (p-1)^2 
   (n-p-1){x_2}^2 \right. \\
   & &  +4 n \left(4 n^3-5 n^2 p-7 n^2+n
   p^2+8 n p+3 n-2 p^2-2 p\right) (2 n-p-1) {x_2} \\
   & & \left.+4 (n-p-1)^2 (3 n-p-1) (2 n-p-1)^2\right) = 0. 
      \end{eqnarray*} 
      If $ n \, {x_2}-2( n- p - 1) \neq 0$, we have 
  \begin{eqnarray*} 
  & & \ \ \left(-n^2
   (p-1)^2 (n+p-1) {x_2}^3    +2 n (p-1) 
   \left(4 n^3-3 n^2 p-9 n^2+2 n p^2+10 n p+4 n  \right. \right. \\
    & &  \left.-4 p^2-4
   p\right){x_2}^2
    + 4 (2 n-p-1) \left(6 n^4+5 n^3
   p-19 n^3-13 n^2 p^2+21 n^2+5 n p^3+9 n p^2 \right. \\
   & & \left. \left. -5 n p-9 n -p^4-2 p^3+2 p+1\right) {x_2} -8 n (n-p-1) (2 n-p-1)^2
   (n+p-1)\right)=0,  \\
   & & \ \   \left(-n^2 (p-1)^2 
   (n-p-1){x_2}^2 +4 n \left(4 n^3-5 n^2 p-7 n^2+n
   p^2+8 n p+3 n-2 p^2-2 p\right)\times \right. \\
   & & \left.   (2 n-p-1) {x_2}+4 (n-p-1)^2 (3 n-p-1) (2 n-p-1)^2\right) = 0. 
      \end{eqnarray*} 
 By taking   the resultant of these polynomials with respect to $x_2$, 
 we get 
  \begin{eqnarray*} 
  & & 
-2048 (n-1)^2 n^6 (p-1)^2 (n-p-1) (n-p) (2 n-p-1)^6  \left(p (n-p)+(n-1)^2\right)\times \\
& & 
   \left(26 n^5-48 n^4
   p-92 n^4+14 n^3 p^2+160 n^3 p+124 n^3+12 n^2 p^3-64
   n^2 p^2-180 n^2 p \right. \\ 
   & & \left. -80 n^2-4 n p^4-7 n p^3+63 n p^2+83
   n p+25 n+3 p^4-2 p^3-16 p^2-14 p-3\right). 
      \end{eqnarray*} 
Now we have    
 \begin{eqnarray*} & & 
 \ \    26 n^5-48 n^4
   p-92 n^4+14 n^3 p^2+160 n^3 p+124 n^3+12 n^2 p^3-64
   n^2 p^2-180 n^2 p  \\ 
   & &    -80 n^2-4 n p^4-7 n p^3+63 n p^2+83
   n p+25 n+3 p^4-2 p^3-16 p^2-14 p-3  \\
   & & 
 = 26 (n-p-1)^5+2 (41 p+19)(n-p-1)^4+2 \left(41 p^2+60 p+8\right) (n-p-1)^3 \\ 
& &  +2 p \left(13 p^2+55 p+30\right) (n-p-1)^2+\left(29 p^3+49 p^2+11
   p-1\right) (n-p-1)+8 p^2 (p+1)
      \end{eqnarray*} 
which is positive for $ 2 \leq p \leq n-2$. Thus  we see that  the resultant is non-zero and we only get   $\displaystyle  {x_2} = \frac{2 ( n- p-1)}{n} $ for a solution of equations (\ref{51})  and  (\ref{61}). From (\ref{41}), we see  $\displaystyle  {x_3} = \frac{  3 n - 2 p- 2}{n} $. Thus we obtain a K\"ahler-Einstein metric in this case. 

Notice that this metric corresponds (up to scale) to the K\"ahler-Einstein metric $g_2$ of (\ref{KE})
  
3)  Let $\displaystyle x_4 = \frac{2(2 n- p -1)}{3 n - 2( p+1) }$. By a similar method we obtain that 
for $ 2 \leq p \leq n-2$,
  $\displaystyle  {x_2} = \frac{2 (n-p-1)}{  3 n - 2 p- 2} $ is the only solution of equations 
  (\ref{51})  and  (\ref{61}), and  from (\ref{41}) we see that 
  $\displaystyle  {x_3} = \frac{n}{  3 n - 2 p- 2}$. Thus we obtain a K\"ahler-Einstein metric in this case.

4)  Let $\displaystyle x_4 = \frac{2 (p-1)}{n+2( p-1)}$. By a similar method we obtain that
for $ 2 \leq p \leq n-2$,
$\displaystyle  {x_2} = \frac{2 (n + p-1)}{n+2 p-2} $ is  the only positive solution of the equations (\ref{51})  and  (\ref{61}) 
for $\displaystyle  \frac{n}{2} \leq p \leq n - 2 $, and from
(\ref{41}) we see  that $\displaystyle  {x_3} = \frac{n}{  n+2 p-2}$. 

Therefore, we obtain a K\"ahler-Einstein metric in all four cases. 

  
  \medskip

  We now denote by $T(x_4)$ the  factor of degree $4$ in the factorization of $Q(x_4)$. Then 
  we can write 
     \begin{eqnarray*} 
   & & T(x_4) =  (n-1) n^2 (3 n-4)(2 n-p-1) {x_4}^4 \\
   & &  + 4 (n-1) n(2 n-p-1) \left(n^2-4 n p-2 p^2+8
   p-2\right){x_4}^3  \\ 
    & &  +2
  (n^5-19 n^4 p+11 n^4+36 n^3 p^2+18 n^3 p-30
   n^3+22 n^2 p^3-130 n^2 p^2+54 n^2 p \\
   & &   +22 n^2-16 n
   p^4-4 n p^3+108 n p^2-68 n p-4 n+16 p^4-16 p^3-16
   p^2+16 p ) {x_4}^2 \\
    & &  -8 (p-1) (n-2 p) (n+p-1)
   \left(n^2-6 n p+2 n+2 p^2+4 p-2\right) {x_4}  \\
   & & 
   +8 (p-1)^2 (n-2 p)^2 (n+p-1). 
    \end{eqnarray*}
  
   The case $n = 2 p$ has been studied in \cite{Chry3}.
   \medskip
   
   We now proceed in two steps.
   \bigskip
   
   \underline{STEP 1.}
   \medskip
  { \bf   { \boldmath We will show that  for $n \geq 4$ and $ 2 \leq p < n/2$ the equation $T(x_4) = 0 $ has no positive solutions.}} 
  
  \medskip
  Note that $T(0) = 8 (p-1)^2 (n-2 p)^2 (n+p-1) > 0$ for $ 2 \leq p < n/2$. 
  
  We have   that
 \begin{eqnarray*}
 & & \frac{d T}{d x_4 }(x_4) =  4 (n-1) n^2 (3 n-4) (2 n-p-1){x_4}^3 \\
 & &  +
12 (n-1) n (2 n-p-1) \left(n^2-4 n p-2
   p^2+8 p-2\right) {x_4}^2 \\ 
   & &+4  \left(n^5-19 n^4
   p+11 n^4+36 n^3 p^2+18 n^3 p-30 n^3+22 n^2 p^3-130
   n^2 p^2+54 n^2 p \right.\\ 
   & & \left.  + 22 n^2-16 n p^4-4 n p^3+108 n
   p^2-68 n p-4 n+16 p^4-16 p^3-16 p^2+16 p \right){x_4} \\
   & &   - 8 (p-1) (n-2 p) (n+p-1) \left(n^2-6
   n p+2 n+2 p^2+4 p-2\right). 
     \end{eqnarray*} 
   Note that the coefficient of ${x_4}^3$ is $ 4 (n-1) n^2 (3 n-4) (2 n-p-1) > 0$.

The polynomial $T(x_4)$ of degree $4$ attains a local minimum at $x_4=u_1$, a local maximum at $x_4=u_2$, and a local minimum
at $x_4=u_3$.

  By evaluating $\displaystyle \frac{d T}{d x_4 }(x_4)$ at the point $\displaystyle \alpha= - \frac{n - 2p}{2 n}<0$, we have that
  
    \begin{eqnarray*}
 & & \frac{d T}{d x_4 }\left( - \frac{n - 2p}{2 n} \right) = 
 \frac{(n-2 p)}{2 n} \left(2 n^5+9 n^4 p-29 n^4+8
   n^3 p^2-33 n^3 p+103 n^3-24 n^2 p^2 \right.\\
  & & \left. -16 n^2 p-112 n^2+8 n p^4+48 n p^3-32 n
   p^2+92 n p+36 n-40 p^4+8 p^3+8 p^2-40 p \right).  
   \end{eqnarray*} 
Since we can write 
  \begin{eqnarray*}  & &  2 n^5+9 n^4 p-29 n^4+8
   n^3 p^2-33 n^3 p+103 n^3-24 n^2 p^2  -16 n^2 p-112 n^2+8 n p^4 \\
  & & +48 n p^3-32 n
   p^2+92 n p+36 n-40 p^4+8 p^3+8 p^2-40 p  \\
   &   & =  2 (n-2 p)^5+(29 p-29) (n-2 p)^4+\left(160 p^2-265 p+103\right) (n-2 p)^3 \\    
    &   &+\left(424 p^3-918 p^2+602  p-112\right) (n-2 p)^2  +\left(552  p^4-1372 p^3+1140 p^2-356 p+36\right)\times \\
 & &   (n-2 p) +288 p^5-768 p^4+704 p^3-256 p^2+32 p
       \end{eqnarray*} 
   we see  that $\displaystyle  \frac{d T}{d x_4 }\left(\alpha  \right) > 0$, thus $u_1<\alpha$

   Also, by evaluating $\displaystyle \frac{d T}{d x_4 }(x_4)$ at the point $\displaystyle x_4=  \beta=\frac{2(p-1)}{n}>0$, we have that
   
    \begin{eqnarray*}
 & & \frac{d T}{d x_4 }\left( \frac{2(p-1)}{n} \right) = 
 -\frac{16 (p-1) (n-p-1)}{n} \left(n^3+3 n^2 p-7 n^2-8 n p^2+4 n p + 8 n \right. \\ 
 & & 
 \left. +2 p^3+6 p^2-6 p-2\right). 
   \end{eqnarray*} 
Since we can write 
  \begin{eqnarray*}  & & n^3+3 n^2 p-7 n^2-8 n p^2+4 n p + 8 n +2 p^3+6 p^2-6 p-2 
  \\  
  &  = & 
 (n-2 p)^3+(9 p-7) (n-2 p)^2+8 (p-1) (2
   p-1) (n-2 p)+2 (p-1)^2 (3 p-1)) > 0, 
    \end{eqnarray*} 
   we see  that $\displaystyle  \frac{d T}{d x_4 }\left(\beta \right) < 0$, thus $\beta <u_3$

   Therefore, the three real solutions $u_1$, $u_2$, $u_3$ of the polynomial $\displaystyle  \frac{d T}{d x_4 }(x_4)$ of degree 3 
  satisfy 
   
   $$\displaystyle u_1 <  \alpha < u_2 < \beta < u_3. $$

    \begin{figure}[htbp]
\begin{minipage}{.48\linewidth}
$T(x_4 )$ \hspace{4.3cm} $\displaystyle \frac{d T}{d x_4 }(x_4)$ 
\vspace{-15pt}
\begin{center}
\includegraphics[width=0.9\linewidth]{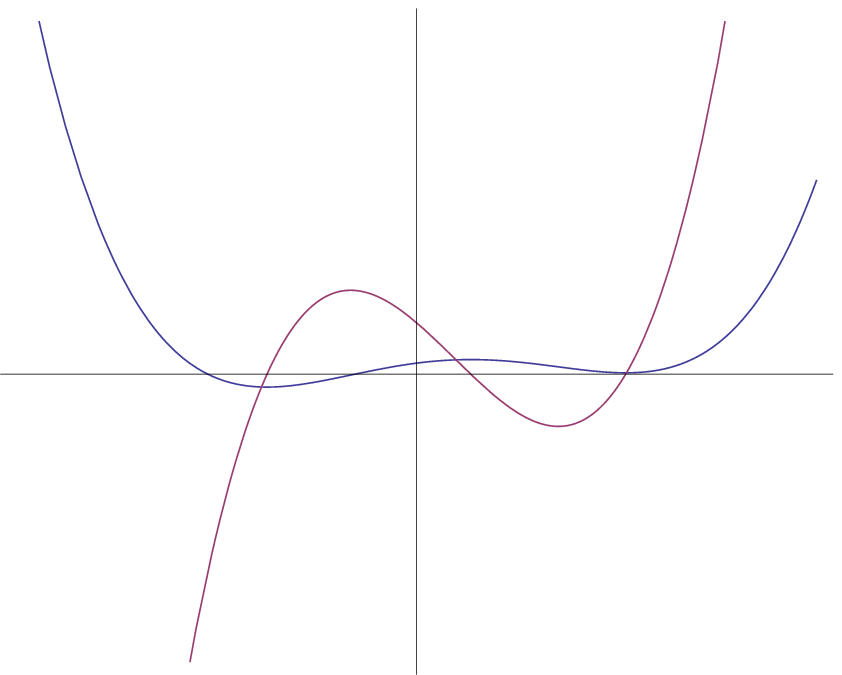}
\end{center}
\end{minipage}
\quad 
\begin{minipage}{0.48\linewidth}
\hspace{4.6cm}$\displaystyle \frac{d T}{d x_4 }(x_4)$ \hspace{0.4cm} $T(x_4 )$ 
\vspace{-30pt}
\begin{center}
\includegraphics[width=0.9\linewidth]{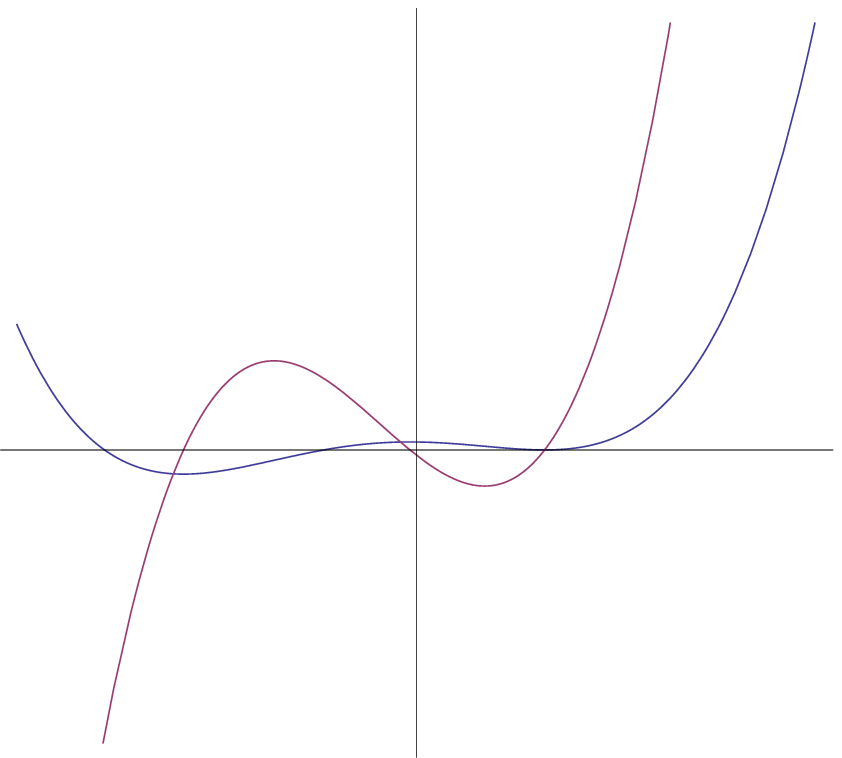}
\end{center}
\end{minipage}
\end{figure}

Since $T(0) > 0$, in order to show that $ T(x_4) > 0 $ for $x_4 >0$,  we need to prove
the following:

{\bf \underline{Claim.}  The local minimum \boldmath{$T(u_3)$} is positive.  }

     
  We show our claim  by dividing  into two cases, namely $ p = 2$ and $p \geq 3$.

  \bigskip
  
{\bf Case 1. \boldmath{$ p = 2$}  }

 The polynomial $T(x_4)$ is given by 
  \begin{eqnarray*} 
  & & T(x_4) =  (n-1) n^2 (2 n-3) (3 n-4) {x_4}^4+4 (n-1) n (2 n-3)
   \left(n^2-8 n+6\right) {x_4}^3 \\
   & & +2 \left(n^5-27 n^4+150 n^3-214 n^2+4
   n+96\right) {x_4}^2 - 8 (n-4) (n+1) \left(n^2-10
   n+14\right) {x_4} \\
   &  & +8 (n-4)^2 (n+1). 
       \end{eqnarray*}
 
Then the local minimum of $T(x_4)$ at $x_4 = u_3$ satisfies 
$2/n < u_3 < 2/n +(2/n)^2$.  

Indeed, it is 
  \begin{eqnarray*}
 & & \frac{d T}{d x_4 }(x_4) = 4 (n-1) n^2 (2 n-3) (3 n-4) {x_4}^3+12 (n-1) n (2 n-3)
   \left(n^2-8 n+6\right) {x_4}^2 \\ 
   & & +4 \left(n^5-27 n^4+150 n^3-214 n^2+4 n+96\right)
  {x_4} - 8 (n-4) (n+1) \left(n^2-10 n+14\right). 
  \end{eqnarray*} 
 Then 
 \begin{eqnarray*}
  & & \frac{d T}{d x_4 }(2/n)= -\frac{16 (n-3) \left(n^3-n^2-16 n+26\right)}{n} \\
   & &  = -\frac{16 (n-3) \left( (n-4)^3+11 (n-4)^2+24 (n-4)+10\right)}{n} < 0 
   \end{eqnarray*} 
 and 
 \begin{eqnarray*}
  & & \frac{d T}{d x_4 }(2/n + (2/n)^2)= \frac{16 \left(n^6+7 n^5-12 n^4-14 n^3-152 n^2+392
   n-192\right)}{n^4} \\
   & &  = \frac{16}{n^4}\left((n-4)^6+31 (n-4)^5+368 (n-4)^4+2194 (n-4)^3 +6848 (n-4)^2 \right. 
   \\
& &   \left.+10536 (n-4)+6240\right) > 0.  
   \end{eqnarray*} 
 
 Also, we have that  
   \begin{eqnarray*}
 & & \frac{d^2 T}{d{ x_4 }^2}(x_4) = 12 (n-1) n^2 (2 n-3) (3 n-4) {x_4}^2+24 (n-1) n (2 n-3)
   \left(n^2-8 n+6\right) {x_4} \\ 
   & & +4 \left(n^5-27 n^4+150 n^3-214 n^2+4 n+96\right) \\
   & & =  12 (n-1) n^2 (2 n-3) (3 n-4)\left( x_4 +  \frac{n^2-8 n+6}{n (3 n-4)} \right)^2+4 (n^5-27 n^4+150 n^3 \\ 
  & &    -214 n^2+4 n+96) - \frac{12 (n-1)  (2 n-3) (n^2-8 n+6)^2}{ (3 n-4)}. 
     \end{eqnarray*} 
 Note that 
 $$\frac{2}{n} - (-  \frac{n^2-8 n+6}{n (3 n-4)}) = \frac{n^2-2 n-2}{n (3 n-4)} = \frac{(n-3)^2+4 (n-3)+1}{n (3 n-4)} > 0 
 $$
 and
   \begin{eqnarray*}
  & &   \frac{d^2 T}{d{ x_4 }^2}(2/n)  = 4 (n-3) (n-2) \left(n^3+2 n^2-26 n+28\right)
  \\
  & & =
4 (n-3) (n-2) \left( (n-4)^3+14 (n-4)^2+38 (n-4)+20\right) > 0. 
    \end{eqnarray*} 
 Hence, the function $T(x_4)$ is concave up for   $x_4 \geq 2/n$, so the local minimum $x_4=u_3$ satisfies
 $2/n <u_3<2/n +(2/n)^2$.
 
 \medskip

We consider the tangent lines of the curve $T(x_4)$ at $x_4 = 2/n$ and $x_4 =2/n + (2/n)^2$, given by the equations 
 \begin{eqnarray*}
 & &  z_1(t) = 
\frac{16 (n-3)^2 (3 n+8)}{n^2}-\frac{16 (n-3) \left(n^3-n^2-16
   n+26\right) }{n} ( t - 2/n ) \\
   & & 
 = -\frac{16 (n-3) \left(\left(n^3-n^2-16 n+26\right) n \  t  -2 n^3-n^2+33 n-28\right)}{n^2}
 \end{eqnarray*}  
and
   \begin{eqnarray*}
   & & z_2(t) = \frac{16 \left(n^6+7 n^5-12 n^4-14 n^3-152 n^2+392 n-192\right)}{n^4}\left(t -\frac{4}{n^2}-\frac{2}{n}\right) \\
   & & +\frac{16 \left(n^7+3 n^5-28 n^4-40 n^3+32 n^2+368 n-192\right)}{n^6}
   \\
   & & 
= \frac{16}{n^6}
 \left(\left(n^6+7 n^5-12 n^4-14 n^3-152 n^2+392 n-192\right) n^2 \ t  -n^7-18 n^6 \right. \ \  \quad  \quad \quad \\ 
   & & \left.-n^5+48 n^4+320 n^3-144 n^2- 816 n+576\right)
 \end{eqnarray*}  
 respectively.  These are shown in the figure. 
   \bigskip
  
\begin{center}
\includegraphics[width=0.4\linewidth]{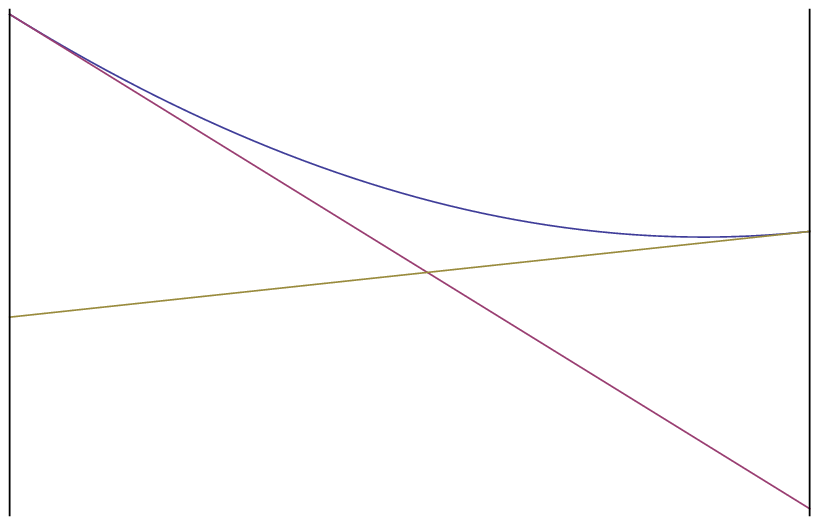} 
\end{center}
\begin{center} \quad \quad \quad $x_4 =2/n$  \hspace{120pt} $x_4 =2/n + (2/n)^2$ 
\end{center}
 
 Let $ (x_0, y_0)$ be their point of intersection given by
 
 \begin{eqnarray*}
   & &  x_0 =  \frac{2 \left(n^8-2 n^7-9 n^6+64 n^5-66 n^4-160 n^3+72 n^2+408
   n-288\right)}{n^2 \left(n^7-3 n^6-6 n^5+62 n^4-92 n^3-152
   n^2+392 n-192\right)}, \\
    & &  y_0 =  
 \frac{16 (n-3)}{n^3 \left(n^7-3
   n^6-6 n^5+62 n^4-92 n^3-152 n^2+392 n-192\right)}\times
   \\
   & &  \left(n^{10}-2 n^9-17 n^8+68 n^7+94 n^6-500 n^5+88
   n^4-5368 n^3+26048 n^2-35808 n+14976\right). 
 \end{eqnarray*} 
 Note that 
  \begin{eqnarray*}
   & & 
 n^{10}-2 n^9-17 n^8+68 n^7+94 n^6-500 n^5+88
   n^4-5368 n^3+26048 n^2-35808 n+14976 \\
   & &  = (n-3)^{10}+28 (n-3)^9+334 (n-3)^8+2252 (n-3)^7+9712 (n-3)^6+29164
   (n-3)^5 \\
   & & +65002 (n-3)^4+102860 (n-3)^3+99479 (n-3)^2+47904 (n-3)+8064 > 0 
  \end{eqnarray*}  for $n \geq 3$.  
  Therefore, the local minimal $T(u_3) $  is greater than $ y_0$, and the claim has been proved.
\medskip

{\bf   Case 2. \boldmath{$ 3\le p < n/2$}.  }

Note that $ n - p \geq p$ and 
\begin{eqnarray*}
T(2(p-1)/n) = \frac{16 (p-1)^2 (n-p-1)^2 \left(n p+n+4( p-1) p\right)}{n^2} > 0. 
 \end{eqnarray*} 
Now we have   
 \begin{eqnarray*}
 & & \frac{d T}{d x_4 }(x_4) =  4 (n-1) n^2 (3 n-4) (2 n-p-1){x_4}^3 \\
 & &  +
12 (n-1) n (2 n-p-1) \left(n^2-4 n p-2
   p^2+8 p-2\right) {x_4}^2 \\ 
   & &+4  \left(n^5-19 n^4
   p+11 n^4+36 n^3 p^2+18 n^3 p-30 n^3+22 n^2 p^3-130
   n^2 p^2+54 n^2 p \right.\\ 
   & & \left.  + 22 n^2-16 n p^4-4 n p^3+108 n
   p^2-68 n p-4 n+16 p^4-16 p^3-16 p^2+16 p \right){x_4} \\
   & &   - 8 (p-1) (n-2 p) (n+p-1) \left(n^2-6
   n p+2 n+2 p^2+4 p-2\right)  
  \end{eqnarray*} 
  and 
   \begin{eqnarray*}
 & & \frac{d T}{d x_4 }( 2(p-1)/n ) = 
 -\frac{16 (p-1) (n-p-1)}{n} \left(n^3+3 n^2 p-7 n^2-8 n p^2+4 n p + 8 n \right. \\ 
 & & 
 \left. +2 p^3+6 p^2-6 p-2\right). 
   \end{eqnarray*} 
Note that 
  \begin{eqnarray*}  & & n^3+3 n^2 p-7 n^2-8 n p^2+4 n p + 8 n +2 p^3+6 p^2-6 p-2 
  \\  
  &  = & 
 (n-2 p)^3+(9 p-7) (n-2 p)^2+8 (p-1) (2
   p-1) (n-2 p)+2 (p-1)^2 (3 p-1)) > 0,
    \end{eqnarray*} 
  thus we see  that $\displaystyle  \frac{d T}{d x_4 }( \beta ) < 0$. 
  
  Let $z_1(t)$ be the  tangent line  of the curve $T(x_4)$ at $x_4 =  \beta$.  This is given by     
   \begin{eqnarray*}
 & &  z_1(t) = 
-\frac{16 (p-1) (n-p-1)}{n} \left(n^3+3 n^2 p-7 n^2-8 n p^2+4 n p+8 n+2 p^3+6 p^2 \right. \\
& & \left. -6 p-2\right) \left(t-\frac{2 (p-1)}{n}\right)+ \frac{16 (p-1)^2 (n-p-1)^2 \left(n p+n+4( p-1) p\right)}{n^2}. 
 \end{eqnarray*}  
  We consider the point $ t_0$  such that  $ z_1(t_0) = 0$. Then we see that 
  \begin{eqnarray*}
 & &  t_0 = \frac{(p-1) \left(2 n^3+7 n^2 p-13 n^2-13 n p^2+2 n p+15 n+12 p^2-8
   p-4\right)}{n \left(n^3+3 n^2 p-7 n^2-8 n p^2+4 n p+8 n+2 p^3+6 p^2-6 p-2\right)}. 
  \end{eqnarray*} 
We will show that $\displaystyle  \frac{d T}{d x_4 }(t_0) > 0$ for  $ 3 \leq p \leq n/2$. 
Indeed, we have   
 \begin{eqnarray*}
 & & \frac{d T}{d x_4 }(t_0) = \frac{4 (p-1) (n-p-1)A(n, p)}{n \left(n^3+3 n^2 p-7 n^2-8 n p^2+4 n p+8 n+2 p^3+6 p^2-6 p-2\right)^3}, 
   \end{eqnarray*}
   where 
   \begin{eqnarray*}
 & &  A(n, p) = 
  n^{12} p-3 n^{12}+15 n^{11} p^2-68 n^{11}
   p+85 n^{11}+73 n^{10} p^3-447 n^{10} p^2 \\
   & & 
   +1135 n^{10} p-985 n^{10}+68
   n^9 p^4-730 n^9 p^3+3590 n^9 p^2-8134 n^9 p+6102 n^9 \\ 
   & & -388 n^8 p^5+1743 n^8 p^4-3118 n^8 p^3-6724 n^8 p^2+28594 n^8 p-22347 n^8-590 n^7 p^6 \\ 
   & & 
   +3284 n^7 p^5-13140 n^7 p^4+38772 n^7 p^3-26930 n^7 p^2-48968 n^7
   p+51156 n^7\\
   & &+1180 n^6 p^7  -3852 n^6 p^6+17728 n^6 p^5-22616 n^6
   p^4-72692 n^6 p^3+123252 n^6 p^2 \\ 
   & & +29464 n^6 p-76048 n^6+961 n^5
   p^8+148 n^5 p^7-20352 n^5 p^6-18940 n^5 p^5 \\
   & & +140330 n^5 p^4-14084 n^5
   p^3-190872 n^5 p^2+29804 n^5 p+75053 n^5-2356 n^4 p^9\\ 
   & & -6225 n^4
   p^8+24308 n^4 p^7+60596 n^4 p^6-94876 n^4 p^5-165630 n^4 p^4+164044
   n^4 p^3 \\ 
   & & +136388 n^4 p^2-67312 n^4 p-49449 n^4+1068 n^3 p^{10}+11644
   n^3 p^9-9136 n^3 p^8 \\
   & &  -59672 n^3 p^7-27216 n^3 p^6+194496 n^3 p^5+28056
   n^3 p^4-177768 n^3 p^3-35740 n^3 p^2 \\
           & & +52804 n^3 p+21464 n^3+32 n^2
   p^{11}-6532 n^2 p^{10}-5992 n^2 p^9+24252 n^2 p^8+53120 n^2 p^7\\ 
   & & -41016 n^2 p^6-130288 n^2 p^5+63304 n^2 p^4+77920 n^2 p^3-7492 n^2 p^2-21416 n^2 p \\ 
   & &  -5892 n^2-64 n p^{12}+1024 n p^{11}+5664 n p^{10}-5120 n
   p^9-11616 n p^8-23168 n p^7\\ 
   & & +44864 n p^6+27264 n p^5-37376 n p^4-12672
   n p^3+5792 n p^2+4480 n p+928 n\\ 
   & & -768 p^{11}-832 p^{10}+2688 p^9+960
   p^8+4608 p^7-12672 p^6+1792 p^5+5248 p^4\\ 
   & & +256 p^3-832 p^2-384 p-64. 
     \end{eqnarray*}
     
We shall show that $ A(n, p) > 0$  for $ 3 \leq p \leq n/2$.  
We can write  $ A(n, p) $  as a polynomial of $y = n -2 p$ of the form  
  $$ A(n,  p) = (p-3) y^{12}+ a_{11} y^{11} + a_{10} y^{10}+ a_{9} y^9+a_{8} y^{8} + a_{7} y^{7}+ a_{6} y^6+a_{5} y^{5} + a_{4} y^{4}+ a_{3} y^3+a_{2} y^{2} + a_{1} y + a_{0}, $$ 
  where   $a_j\ (j=0,\dots ,11)$ can be written as follows: 
  \begin{eqnarray*}
  & & a_{11} = 39 (p-3)^2+94 (p-3)+16 \\
    & & a_{10} = 667 (p-3)^3+3268 (p-3)^2+4604 (p-3)+1424 \\
      & & a_{9} = 6588 (p-3)^4+49146 (p-3)^3+131552 (p-3)^2+146040 (p-3)+53568 \\
     & & a_{8} =41696 (p-3)^5+420063 (p-3)^4+1666118 (p-3)^3+3239144 (p-3)^2 \\ & & +3068032 (p-3)+1121664 \\
    & & a_{7} = 177618 (p-3)^6+2258984 (p-3)^5+11898662 (p-3)^4+33203396
   (p-3)^3\\ & & +51731888 (p-3)^2+42627136 (p-3)+14495232 \\
 & & a_{6} =   521336 (p-3)^7+8010196 (p-3)^6+52618296 (p-3)^5+191551956
   (p-3)^4 \\ 
   & &+417348472 (p-3)^3+544194848 (p-3)^2+393195520
   (p-3)+121432064 \\
     & & a_{5} = 1062393 (p-3)^8+19117036 (p-3)^7+150329840 (p-3)^6+674768512
   (p-3)^5 \\
    & & +1890947640 (p-3)^4+3387906256 (p-3)^3+3789854976
   (p-3)^2 \\  
   & & +2420175872 (p-3)+675510272 \\
    & & a_{4} = 1493910 (p-3)^9+30767865 (p-3)^8+281434708 (p-3)^7+1500619596
   (p-3)^6 \\ 
   & &+5140194384 (p-3)^5+11730160160 (p-3)^4+17833993024
   (p-3)^3 \\ 
   & & +17419011328 (p-3)^2+9918241792 (p-3)+2508337152 \\
     & & a_{3} = 1416852 (p-3)^{10}+32818860 (p-3)^9+341869872 (p-3)^8+2109020632
   (p-3)^7 \\
   & & +8532907744 (p-3)^6+23658308832 (p-3)^5+45523459968
   (p-3)^4 \\ 
   & & +60028498688 (p-3)^3+51913028096 (p-3)^2+26587561984
   (p-3)+6123782144 \\
    & & a_{2} = 862488 (p-3)^{11}+22162788 (p-3)^{10}+258695208 (p-3)^9+1810579704
   (p-3)^8 \\ 
   & &  +8442449008 (p-3)^7+27537781712 (p-3)^6+64116833984
   (p-3)^5 \\ 
   & & +106560650432 (p-3)^4+123887801600 (p-3)^3+95957073920
   (p-3)^2 \\
   & & +44563972096 (p-3)+9401008128\\
 & & a_{1} = 303264 (p-3)^{12}+8551008 (p-3)^{11}+110430432 (p-3)^{10}+863710128
   (p-3)^9 \\ 
   & & +4556601456 (p-3)^8+17082048928 (p-3)^7+46660844352
   (p-3)^6 \\
    & & +93574409856 (p-3)^5+136732708864 (p-3)^4+141973649408
   (p-3)^3 \\ & & +99432382464 (p-3)^2+42173857792 (p-3)+8192524288 \\
    & & a_{0} = 46656 (p-3)^{13}+1430784 (p-3)^{12}+20235744 (p-3)^{11}+174764304
   (p-3)^{10} \\
   & & +1028302272 (p-3)^9+4352962512 (p-3)^8+13638809216
   (p-3)^7\\
   & & +32024909952 (p-3)^6+56352955904 (p-3)^5+73394750720
   (p-3)^4 \\
   & & +68769538048 (p-3)^3+43897815040 (p-3)^2+17110138880
   (p-3)+3075473408.
      \end{eqnarray*}
  We see that  the  coefficients $a_{j}$  ($j =0, \dots ,11$)  are positive for $p \geq 3$, which means that $ A(n, p) > 0$  for $ 3 \leq p <  n/2$.  
  Therefore, $\displaystyle  \frac{d T}{d x_4 }(t_0) > 0$ for  $ 3 \leq p \leq n/2$. 
  \medskip

 Now we compute $\displaystyle  \frac{d^2T}{d {x_4 }^2}(x_4)$. We see that 
  \begin{eqnarray*}
  & & 
  \frac{d^2T}{d {x_4 }^2}(x_4) = 12 (n-1) n^2 (3 n-4)  (2 n-p-1){x_4}^2 \\
  & &  +
 24 (n-1) n (2 n-p-1) \left(n^2-4
   n p-2 p^2+8 p-2\right)  {x_4}  \\
   & & +4
   \left(n^5-19 n^4 p+11 n^4+36 n^3 p^2+18
   n^3 p-30 n^3+22 n^2 p^3-130 n^2 p^2+54
   n^2 p+22 n^2  \right.\\
   & &  \left.-16 n p^4-4 n p^3+108 n
   p^2-68 n p-4 n+16 p^4-16 p^3-16 p^2+16
   p\right)\\
   & & = 12 (n-1) n^2 (3 n-4)  (2 n-p-1)\left(x_4 + \frac{n^2-4 n p-2 p^2+8 p-2}{n (3 n-4)} \right)^2\\
   & & +4
   \left(n^5-19 n^4 p+11 n^4+36 n^3 p^2+18
   n^3 p-30 n^3+22 n^2 p^3-130 n^2 p^2+54
   n^2 p+22 n^2  \right.\\
   & &  \left.-16 n p^4-4 n p^3+108 n
   p^2-68 n p-4 n+16 p^4-16 p^3-16 p^2+16
   p\right)\\
   & & - 12 (n-1)  (2 n-p-1) \frac{(n^2-4 n p-2 p^2+8 p-2)^2}{ (3 n-4)}. 
    \end{eqnarray*}
    
  Note that 
   \begin{eqnarray*}
   & & 
 \beta - (- \frac{n^2-4 n p-2 p^2+8 p-2}
 {n (3 n-4)}) = \frac{n^2+2 n p-6 n-2 p^2+6}{n (3 n-4)} \\
   & &  =  \frac{(n-2 p)^2+6 (p-1) (n-2 p)+6 (p-1)^2}{n (3 n-4)} > 0,  
    \end{eqnarray*}
    
    and 
    \begin{eqnarray*}
 & &     \frac{d^2T}{d {x_4 }^2}\left( \beta \right)  = 4 (n-p-1) \left(n^4+6 n^3 p-12 n^3+6 n^2
   p^2-60 n^2 p+66 n^2-8 n p^3 \right.\\
   & & \left. +32 n p^2+48
   n p-80 n+8 p^3-40 p^2+8 p+24\right)\\
   & & = 4 (n-p-1)\left((n-2 p)^4+2 (7 p-6) (n-2 p)^3+66
   (p-1)^2 (n-2 p)^2 \right. \\
   & & \left. +8 (p-1) \left(15 p^2-29 p+10\right) (n-2
   p)+8 (p-1) (3 p-1)
   \left(3 p^2-7 p+3\right) \right) > 0. 
         \end{eqnarray*}
Therefore,  the function $T(x_4)$ is concave up for   $x_4 \geq 2(p-1)/n$. 

\bigskip

\begin{minipage}{.38\linewidth}
\begin{center}
\includegraphics[width=0.99\linewidth]{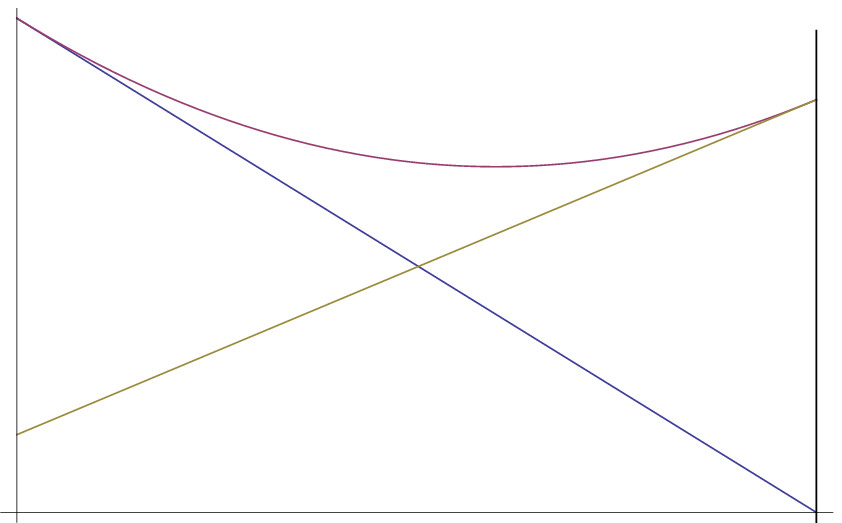}
\end{center}
$x_4 = \beta$ \hspace{85pt} $x_4 = t_0$ 
\end{minipage}
\quad 
\begin{minipage}{0.6\linewidth}
Consider the tangent line $l_1$ of the curve $T(x_4)$ at $x_4 = \beta$, which intersects $x$-axis at a point $t_0 $,
and  the tangent line $l_2$ of the curve $T(x_4)$ at $x_4 = t_0$.
Since  $\displaystyle \frac{d T}{d x_4 }(\beta) < 0$ and $\displaystyle \frac{d T}{d x_4 }(t_0) > 0$, 
the tangent lines $l_1, l_2$  intersect at a point $ (x_0, y_0)$ with $ y_0 > 0$. 
Since $T(x_4)$ is concave up,   we see that the curve $ ( x_4, T(x_4) ) $ $(\beta \leq x_4 \leq t_0 )$ lies inside the triangle given by the
three points $( \beta, T( \beta)$, $ (x_0, y_0)$ and $( t_0, T(t_0) )$. \end{minipage}

\medskip
 Since the point $(u_3, T(u_3) )$ is inside of this  triangle, it follows that the local minimum  $ T(u_3)$  is  greater than $y_0 > 0$,
 and the claim has also been shown in this case.

\bigskip

\underline{STEP 2.}
{ \bf   { \boldmath We consider the case that  $n \geq 4$ and $n-2 \geq p > n/2$. }}

\medskip

This reduces to case STEP 1 as follows.

We consider   the resultant of $ F(x_2, x_4)$ and  $ G(x_2, x_4)$ with respect to $x_4$,  which is a polynomial of $x_2$  (instead of $x_4$),   and  we denote  this resultant by $R(x_2)$.  By factorizing $R(x_2)$ we have that 
  \begin{eqnarray*} 
   & & 
R(x_2) = 128 (n-1)^6 (p-1)^4 (n-p-1)^2{x_2}^8 (n{x_2}-2 n-2 p+2) (n{x_2}-2 n+2 p+2)\times \\
 & &    (3 n{x_2}-2 n-2 p{x_2}+2 p-2{x_2}+2) (n
  {x_2}-2 n+2 p{x_2}-2 p-2{x_2}+2) \\
  & & \left(3 n^5 {x_2}^4-20 n^5{x_2}^3+48 n^5{x_2}^2-48 n^5{x_2}+16
   n^5+3 n^4 p{x_2}^4+12 n^4 p{x_2}^3-110 n^4 p
  {x_2}^2 \right.\\ 
& &  +200 n^4 p{x_2}-104 n^4 p-10 n^4{x_2}^4+72 n^4
  {x_2}^3-178 n^4{x_2}^2+168 n^4{x_2}-40 n^4+24 n^3 p^2
  {x_2}^3 \\
  & &  +12 n^3 p^2{x_2}^2-248 n^3 p^2{x_2}+256 n^3
   p^2-7 n^3 p{x_2}^4-44 n^3 p{x_2}^3+380 n^3 p
  {x_2}^2-640 n^3 p{x_2} \\
  & &  +224 n^3 p+11 n^3{x_2}^4-92 n^3
  {x_2}^3+232 n^3{x_2}^2-200 n^3{x_2}+32 n^3-8 n^2 p^3
  {x_2}^3+84 n^2 p^3{x_2}^2 \\
 & &   +48 n^2 p^3{x_2}-296 n^2
   p^3-48 n^2 p^2{x_2}^3-92 n^2 p^2{x_2}^2+736 n^2 p^2
  {x_2}-440 n^2 p^2+4 n^2 p{x_2}^4 \\
 & &   +56 n^2 p{x_2}^3-444
   n^2 p{x_2}^2+656 n^2 p{x_2}-152 n^2 p-4 n^2{x_2}^4+48
   n^2{x_2}^3-124 n^2{x_2}^2 \\ 
   & &  +96 n^2{x_2}-8 n^2-32 n p^4
  {x_2}^2+80 n p^4{x_2}+160 n p^4+8 n p^3{x_2}^3-120 n
   p^3{x_2}^2-256 n p^3{x_2} \\
   & &  +352 n p^3+24 n p^2
  {x_2}^3+120 n p^2{x_2}^2-576 n p^2{x_2}+224 n p^2-24 n
   p{x_2}^3+200 n p{x_2}^2\\
   & &   -256 n p{x_2}+32 n p-8 n
  {x_2}^3+24 n{x_2}^2-16 n{x_2}-32 p^5{x_2}-32
   p^5+32 p^4{x_2}^2-96 p^4\\
   & &  \left.+32 p^3{x_2}^2+128 p^3{x_2}-96
   p^3-32 p^2{x_2}^2+128 p^2{x_2}-32 p^2-32 p{x_2}^2+32 p
  {x_2}\right).
    \end{eqnarray*} 
    
    We denote by $S(x_2)$ the factor  of degree $4$ in the above factorization. Then 
  we can write 
     \begin{eqnarray*} 
   & & S(x_2) =  (n-1) n^2 (3 n-4)(n+p-1){x_2}^4 \\
   & & -4 (n-1) n (n+p-1) \left(5 n^2-8 n p-8 n+2 p^2+8
   p+2\right){x_2}^3 \\
   & & +2 \left(24 n^5-55 n^4 p-89 n^4+6 n^3 p^2+190 n^3 p+116
   n^3+42 n^2 p^3-46 n^2 p^2-222 n^2 p \right. \\
   & & \left. -62 n^2-16 n p^4-60 n p^3+60 n
   p^2+100 n p+12 n+16 p^4+16 p^3-16 p^2-16 p\right){x_2}^2
   \\
   & & -8(n-2 p) (n-p-1) (2 n-p-1) \left(3 n^2-2 n p-6
   n-2 p^2+4 p+2\right){x_2}\\
   & &  +8 (n-2 p)^2
   (n-p-1)^2 (2 n-p-1). 
    \end{eqnarray*}
    
    If we  replace  $p$ with $ n-p$ in the polynomial $S(x_2)$,  we get exactly  the same polynomial as $T(x_2)$, and thus we see that the equation $S(x_2) = 0 $ has no positive solutions for  $n-2 \geq p > n/2$. 
  
 \end{proof}

\noindent
The Main Theorem now follows from Propositions \ref{solutions} and \ref{kahler}.

\section{The isometry problem}
In this section we  study  the isometry problem for the new homogeneous Einstein metrics of $M=SO(2n)/U(p)\times U(n-p)$, corresponding to the pairs $(n, p)$ which are presented in the Main Theorem.  Recall that  when   $n=2p$, it was proved in \cite{Chry3} that the  non-K\"ahler   homogeneous Einstein metrics of the form $g=(1, x_2, 1, x_2)$, where $x_2$  is  given by  part $(a)$ of (\ref{400}),   are   not isometric.  However for the special case of $2\leq p\leq 6$, the isometry problem for the remaining two new Einstein metrics $g=(1, x_2, 1, x_4)$, where $x_2$ and $x_4$ are determined by part $(b)$ of   (\ref{400}), and $(\ref{3})$  respectively, has not been studied yet.\footnote{Note that  the first two   non-K\"ahler  Einstein metrics on $M=SO(4p)/U(p)\times U(p)$, were obtained in \cite[Theorem.~8]{Chry3} with respect to the normalization $g=(x_1, 1, x_1, 1)$. For the special case   $2\leq p\leq 6$ the  new  Einstein metrics are given with respect to the normalization $g=(x_1, 1, x_1, x_4)$.}   

    Let us recall the method used in \cite{Chry3}.  For any $G$-invariant Einstein metric $g=(x_1, x_2, x_3, x_4)$ on  $M=SO(2n)/U(p)\times U(n-p)$ (with  $2\leq p\leq n-2$)
 we determine a (normalized) scale invariant given by $H_{g}=V_{g}^{1/d}S_g$, where
  $S_g$ is the scalar curvature of the given metric $g$, 
   $V_{g}=\prod_{i=1}^{4}x_{i}^{d_i}$ is the volume of
   $g$ and $d=\sum_{i=1}^{4}d_{i}=\dim M$. In particular,
 the   scalar curvature of $g$ is given by 
\[
S_g=\frac{1}{2}\sum_{i=1}^{4}\frac{d_i}{x_i} - \frac{[123]}{2}(\frac{x_1}{x_2x_3} + \frac{x_2}{x_1x_3} + \frac{x_3}{x_1x_2}) - 
 \frac{[134]}{2}(\frac{x_1}{x_3x_4} + \frac{x_3}{x_1x_4}+\frac{x_4}{x_1x_3})
 \]
where $d_{i}$ and $[123]$, $[234]$ are given in Section 2. Note that $S_{g}$  is a homogeneous  polynomial of degree $-1$  on the variables $x_{i}$, and the volume $V_{g}$ is a  monomial of degree $d$. 
 Thus  $H_{g}=V_{g}^{1/d}S_g$ is a homogeneous polynomial of degree 0, and it is invariant under a common scaling of the variables $x_i$.  If two metrics are isometric then they have the same scale invariant, so if
  the scale invariants $H_{g}$ and $H_{g'}$ 
 are different, then the metrics $g$ and $g'$ can not  be isometric. 
If $H_{g}=H'_{g}$
 we can not draw an immediate decision and 
  conclude if the metrics $g$ and $g'$ are isometric or not. 
  Finally, K\"ahler-Einstein metrics which correspond to 
   equivalent invariant complex structures on $M$ are isometric (cf. \cite{Chry3}).

In order to detect which pairs of Einstein metrics in the Main Theorem are isometric or not, first we need to give their approximate values.  Note that the non-K\"ahler Einstein metrics   are of the form $g=(1, x_2, 1, x_4)$, where $x_2$ is obtained by solving    equation $H_{n,p}(x_{2})=0$ (see $(\ref{4})$),  if we first substitute the corresponding values of $n$ and $p$.  Next, $x_4$ is easily obtained from (\ref{3}). In the following table we present   the case of $n\neq 2p$.  

\medskip
{\footnotesize
\begin{center}
{\bf Table 1} {\small Approximate values of Einstein metrics on $M$ for   pairs $(n, p)$ with $n\neq 2p$}
\end{center}
\smallskip
  \begin{center}
\begin{tabular}{lllll}
 Pair & Einstein metrics \\
\hline 
\smallskip
$(n, p)$ & $g_1=(1, x_2, 1, x_4)$ & $g_2=(1, x_2, 1, x_4)$ & $g_3=(1, x_2, 1, x_4)$ & $g_4=(1, x_2, 1, x_4)$ \\
 \thickline 
$(7, 4)$ & $(1, 0.4661, 1, 0.7256)$ & $(1, 0.6614, 1, 1.7636)$ & $(1, 1.4144, 1, 1.3999)$ & $(1, 1.5722, 1, 1.0631)$ \\ 
$(7, 3)$ & $(1, 0.7256, 1, 0.4661)$ & $(1, 1.7636, 1, 0.6614)$ & $(1, 1.3999, 1, 1.4144)$ & $(1, 1.0631, 1, 1.5722)$\\
$(6, 4)$ & $(1, 0.2680, 1, 0.8876)$ & $(1, 0.3631, 1, 1.9057)$ & $(1, 1.3782, 1, 1.5645)$ & $(1, 1.5461, 1, 1.1658)$\\
$(6, 2)$ & $(1, 0.8876, 1, 0.2680)$ & $(1, 1.9057, 1, 0.3631)$ & $(1, 1.5645, 1, 1.3782)$ & $(1, 1.1658, 1, 1.5461)$\\
$(5, 3)$ & $(1, 0.3241, 1, 0.6954)$ & $(1, 0.4361, 1, 1.8876)$ & $(1, 1.4331, 1, 1.5883)$ & $(1, 1.6922, 1, 0.8952)$\\
$(5, 2)$ & $(1, 0.6954, 1, 0.3241)$ & $(1, 1.8876, 1, 0.4361)$ & $(1, 1.5883, 1, 1.4331)$ & $(1, 0.8952, 1, 1.6922)$\\
\hline
\end{tabular}
\end{center}}

\medskip
Note that  the values $(7, 4)$ and  $(7, 3)$, $(6, 4)$ and $(6, 2)$, $(5, 3)$ and $(5, 2)$, determine the quotients
{\footnotesize{
\[
 \begin{tabular}{ll}
 $M_1=SO(14)/U(4)\times U(3)$, & $M^{1}=SO(14)/U(3)\times U(4)$, \\
$M_2= SO(12)/U(4)\times U(2)$, & $M^{2}= SO(12)/U(2)\times U(4)$, \\ 
$M_3= SO(10)/U(3)\times U(2)$, & $M^{3}= SO(10)/U(2)\times U(3)$,
\end{tabular}
\]
}}

\smallskip
\noindent respectively. In particular, as we can see from Table 1, the Einstein metrics on  $M^{i}$ are obtained from the Einstein metrics on    $M_{i}$, by a permutation of the components $x_2, x_4$, for any $i=1, 2, 3$, and conversely.\footnote{In general, the flag manifolds  $SO(2n)/U(n-p)\times U(p)$ and $SO(2n)/U(p)\times U(n-p)$ are
isometric via an element of the Weyl group of $G$.}  Thus  we obtain the isometries $M_1\cong M^{1}$, $M_{2}\cong M^{2}$ and $M_{3}\cong M^{3}$. This result is also obtained from  Table 2, where we give the values of the corresponding scale invariants for the Einstein metrics $g_1, g_2, g_3$ and $g_4$.

 Also, from Table 2 we  easily conclude that all non K\"ahler  invariant Einstein metrics on   $M_1\cong M^{1}$, $M_2\cong M^2$ and $M_3\cong M^{3}$  are   not isometric, since for any case it is $H_{g_1}\neq H_{g_2}\neq H_{g_3}\neq H_{g_4}$.  This  completes the examination of the case $n\neq 2p$.

 \medskip
{\footnotesize
\begin{center}
{\bf Table 2} {\small The values of the corresponding scale invariants}
\end{center}
 \smallskip
  \begin{center}
\begin{tabular}{lllllll}
   \hline 
\smallskip
Scale invariants  & $(7, 4)$ & $(7, 3)$ & $(6, 4)$ & $(6, 2)$ & $(5, 3)$  & $(5, 2)$\\
  \thickline
$H_{g_1}$ & $25.2814$ & $25.2814$ & $17.9698$ & $17.9698$  & $12.4373$ & $12.4373$\\
$H_{g_2}$ & $25.5264$ & $25.5264$ & $18.1243$ & $18.1243$ & $12.6088$ & $12.6088$\\
$H_{g_3}$ & $25.6020$ & $25.6020$ & $18.2540$ & $18.2540$ & $12.7050$ & $12.7050$\\
$H_{g_4}$ & $25.5943$ & $25.5943$ & $18.2446$ & $18.2446$ & $12.6700$ & $12.6700$\\
\hline
\end{tabular}
\end{center}}

\medskip
For the special case $n=2p$ with $2\leq p\leq 6$,   the scale invariants corresponding to the   new non-K\"ahler Einstein metrics on $M=SO(4p)/U(p)\times U(p)$ given by $g=(1, x_2, 1, x_4)$, where $x_2$ and $x_4$ are determined by part $(b)$ of   (\ref{400}), and $(\ref{3})$, respectively, are equal.  However, for
  $$       x_2 = \frac{2 p (2 p-1)-\sqrt{2} \sqrt{-p \left(p^3-7 p^2+5 p-1\right)}}{p (3 p-1)},$$   $x_4$ is given by 
       $$ x_4 = \frac{2 p (2 p-1) + \sqrt{2} \sqrt{-p \left(p^3-7 p^2+5 p-1\right)}}{p (3 p-1)},$$ 
       and for
       $$       x_2 = \frac{2 p (2 p-1)+\sqrt{2} \sqrt{-p \left(p^3-7 p^2+5 p-1\right)}}{p (3 p-1)},$$   $x_4$ is given by 
       $$ x_4 = \frac{2 p (2 p-1) - \sqrt{2} \sqrt{-p \left(p^3-7 p^2+5 p-1\right)}}{p (3 p-1)}.$$
        Thus these two Einstein metrics on $M$  are isometric.    

\end{document}